%% file: ex_article.tex
\documentclass[hidelinks,onefignum,onetabnum]{siamart251216}

\input{ex_shared}

\ifpdf
\hypersetup{
  pdftitle={A sGS-APALM for GNEPs in Banach spaces},
  pdfauthor={Defeng Sun, Hailing Wang, Wei Zhao}
}
\fi

\newsiamremark{assumption}{Assumption}

\begin{document}

\maketitle

\begin{abstract}
In this paper, we study a class of monotone generalized Nash equilibrium problems (GNEPs) with jointly linear constraints. The players' strategy spaces are real Hilbert spaces, while the joint constraint is formulated in a Banach space. To solve such problems, we propose a novel symmetric Gauss–Seidel (sGS) based alternating proximal augmented Lagrangian method (sGS-APALM) which incorporates newly designed quadratic surrogates. In contrast to existing regularization and ALM-type methods, the proposed method avoids solving coupled Nash equilibrium subproblems at each iteration and instead updates the players' strategies alternately by solving a sequence of unconstrained quadratic programs. Moreover, unlike many existing splitting-based methods, our global convergence analysis and convergence rate estimation require only monotonicity and Lipschitz continuity of the pseudo-gradient mapping, without imposing stronger assumptions such as strong monotonicity or cocoercivity. Finally, we apply the method to a class of risk-neutral PDE-constrained GNEPs with joint state constraints, and preliminary numerical results demonstrate its efficiency and effectiveness.
\end{abstract}



\section{Introduction}


In various practical applications, multiple decision-makers seek to minimize their individual objectives simultaneously while being affected by the strategies of others through both coupled cost functions and shared constraints. This scenario gives rise to a generalized Nash equilibrium problem (GNEP). In contrast to standard optimization problems, GNEPs contain a collection of player-specific problems where both the feasible region and the objective depend on the decisions of other players. This complicated structure makes computation challenging and motivates us to design efficient algorithms for GNEPs.

In this paper, we focus on a class of monotone GNEPs with $N$ players, which arise in, e.g., multi-objective optimal control \cite{MTA-15}, shape optimization \cite{PD-10}, and electricity dispatch \cite{WHW-25}. Player $\nu$ ($\nu=1,2,\ldots,N$) controls a variable $x_\nu\in X_\nu$, where $X_\nu$ is a real Hilbert space. We define $X:=X_1\times X_2\times\cdots\times X_N$ as the Hilbert space of all players' strategies. For $x=(x_1,x_2,\ldots,x_N)\in X$, let $x_{-\nu}$ denote the strategies of all players except the $\nu-$th one, and let $X_{-\nu}=\prod_{i\neq\nu}X_i$ represent the corresponding strategy space. The $\nu-$th player intends to solve the optimization problem
\begin{equation}\label{sGNEP}\tag{P}
\begin{aligned}
\min_{x_\nu\in X_\nu}\quad &J_\nu(x_\nu,x_{-\nu})+\phi_\nu(x_\nu) \\
\mbox{s.t.}\quad &A_1x_1+{A}_2x_2+\cdots+{A}_Nx_N\in K.
\end{aligned}
\end{equation}
Here, $J_\nu: X\to\mathbb{R}$ is continuously Fr\'echet differentiable, and $\phi_\nu: X_\nu\to\mathbb{R}\cup\{+\infty\}$ is proper, convex, and lower semicontinuous. Moreover, the set $K\subseteq Y$ is closed and convex, where $Y$ is a Banach space, and ${A}_\nu\in\mathcal{L}(X_\nu,Y)$, for $\nu=1,2,\ldots,N$, with $\mathcal{L}(X_\nu,Y)$ denoting the space of bounded linear operators that map from $X_\nu$ to $Y$. 
Let $P_J(x) := (D_{x_1}J_1(x), \ldots, D_{x_N}J_N(x))$ be the pseudo-gradient mapping, where $D_{x_\nu}J_\nu(x)$ denotes the partial Fr\'echet derivative of $J_\nu$ with respect to $x_\nu$. Throughout this paper, we assume that the mapping $P_J$ is monotone. 
We refer to this model as Problem (\ref{sGNEP}). 
This formulation covers both finite-dimensional and infinite-dimensional settings, as well as stochastic models where the underlying probability distributions are known.
Given this general formulation and the extensive literature on the GNEP (\ref{sGNEP}), we review the related work along two strands: (i) finite-dimensional deterministic/stochastic formulations, and (ii) infinite-dimensional extensions. 

In the finite-dimensional deterministic case, the problem reduces to the monotone GNEPs studied in \cite{FP-07}. Then, numerous algorithms and theoretical results are available for monotone GNEPs; see the surveys \cite{FC-10,AMK-14} for standard algorithmic frameworks, Chapter 12 of \cite{FP-10} and Chapter 11 of \cite{CP-21} for connections with quasi-variational inequalities, as well as some recent progress on exact penalty methods \cite{BP-22} and augmented Lagrangian methods \cite{CD-16,MTM-23}. Nevertheless, substantial computational difficulties remain, especially for large-scale problems with coupled constraints and nonsmooth structures.

Finite-dimensional stochastic Nash equilibrium problems (sNEPs) have also been extensively studied; see \cite{RTRJ-17,PSS-17,LSPS-20} for recent developments. Among these works, one line is to reformulate sNEPs as a two-stage stochastic variational inequality (SVI) \cite{WC-23,ZSX-19,RTS-19,RTRJ-17}. Under monotonicity of the underlying operator, the obtained SVI can be solved by the progressive hedging algorithm (PHA). From the numerical results reported in \cite{ZSX-19,RTS-19}, we note that PHA is computationally efficient when the number of scenarios is moderate. However, a direct extension of this framework to stochastic GNEPs of the form \eqref{sGNEP} is unattractive in practice: each PHA iteration would require solving one Nash equilibrium problem for each scenario, which is computationally expensive for the large-scale case, and the total computational cost grows rapidly with the number of scenarios.

Another major line of research for sNEPs is based on best-response (BR) schemes \cite{FGDJ-14,PSS-17,LSPS-20,SPS-16}. The main idea is to decouple the original sNEP into a series of classical optimization problems in a Gauss-Seidel or Jacobi manner. As further shown in \cite{LSPS-20}, if the best-response mapping is continuous and contractive, the resulting BR iteration converges to a Nash equilibrium at a linear rate. Although the BR-based methods behave efficiently on sNEPs, several difficulties arise when extending them to solve stochastic GNEPs. The presence of joint constraints fundamentally changes the structure of the problem, so that BR methods cannot be applied directly. As discussed in \cite{BP-22} and Section 11.3 of \cite{CP-21}, we can employ an exact penalty method to convert the GNEP into a penalized NEP. However, this introduces practical challenges. First, a suitable lower bound for the exact penalty parameter is unknown, which in practice leads to a double-loop iteration scheme. Moreover, the penalized NEP involves non-smooth penalty terms, which complicates the implementation of BR iteration. These issues motivate us to develop an easy-to-implement single-loop algorithm in this work, which can be used to address stochastic GNEPs.

In contrast to the extensive theory and algorithms available for finite-dimensional GNEPs, work on the infinite-dimensional Problem (\ref{sGNEP}) remains limited (as noted in the introduction of \cite{GHS-23}). Nevertheless, many applications of GNEPs are naturally posed in infinite-dimensional spaces, including shape optimization \cite{PD-10,CHPR-11}, multiple objective optimal control \cite{GHS-23,MT-13}, and differential games \cite{MDB-18}. Motivated by these, our aim is to propose a novel and efficient algorithm for the GNEP (\ref{sGNEP}), covering finite- and infinite-dimensional cases, as well as deterministic and stochastic settings.

In the deterministic and infinite-dimensional settings, several methods have been proposed. In \cite{MT-13}, a Moreau-Yosida regularization approach was developed for a class of multi-objective PDE-constrained optimal control problems, and was subsequently extended to monotone GNEPs in Banach spaces \cite{MTA-15}. The key idea of the regularization method is to penalize the joint constraints in the objective and then solve the resulting NEPs by some well-developed methods, such as semismooth Newton methods \cite{MR-09}. As shown in \cite{MTA-15}, the solution of the regularized problem converges to the exact solution of the original problem only if the regularization parameter tends to infinity or zero. However, from the aspect of computation, this results in extremely ill-conditioned NEP subproblems. Worse still, each NEP is coupled with all players and formulated in general Banach spaces, which makes its numerical solution non-trivial. Although some path-following strategies have been proposed to alleviate the ill-conditioning issue, the overall computational cost remains high. In addition to regularization methods, augmented Lagrangian methods (ALMs) have also been applied to solve monotone GNEPs \cite{CD-19,CV-19}. The ALMs are mainly featured by the idea of multiplier-penalties. Compared with regularization methods, ALM-type methods may generate less ill-conditioned subproblems. However, each iteration still requires solving a coupled NEP, and thus faces similar computational challenges as in regularization methods, especially when the penalty parameter becomes large.

To avoid the coupled and ill-conditioned NEP subproblems inherent in regularization and ALM-type methods, various splitting methods have been developed for monotone GNEPs under additional regularity conditions; see \cite{ECp-18,EC-21} for infinite-dimensional settings and \cite{YP-17,YP-19,GPS-22,RLZ-24} for finite-dimensional cases. These methods are mainly based on: (i) extensions of multi-block ADMM \cite{EC-21}, (ii) regularized Jacobi-ADMM \cite{ECp-18}, and (iii) forward-backward splitting frameworks \cite{YP-17,YP-19,GPS-22,RLZ-24}. For ADMM-type methods \cite{ECp-18,EC-21}, each iteration replaces the original NEP subproblem with a sequence of single-objective optimization problems. While this avoids solving the NEP subproblem directly, it typically requires adding sufficiently large proximal terms to guarantee convergence, which slows down practical progress. Moreover, for general models the resulting subproblems can be nonlinear (and sometimes non-smooth), so that inner iterative solvers are often required; hence each outer iteration can be computationally expensive and the overall complexity remains high. In contrast, forward–backward splitting methods \cite{YP-17,YP-19,GPS-22,RLZ-24} are relatively simple to implement, but their step sizes are typically quite restrictive, as the choice must consider both the objective-related Lipschitz information and the coupling constraints. Consequently, their practical convergence rate can be slow. Indeed, the numerical comparisons in \cite{EC-21} have shown that forward–backward schemes often converge more slowly in practice than ADMM-type methods. 

Apart from the aforementioned splitting-based methods, a symmetric Gauss-Seidel (sGS) based majorized-ALM method \cite{WWW-25} has recently been designed for deterministic monotone GNEPs under additional regularity assumptions. Each iteration only needs to solve a series of unconstrained quadratic optimization problems, thereby improving practical efficiency compared with existing splitting-based methods. Nevertheless, an important theoretical limitation still remains: most existing splitting-type frameworks (including \cite{WWW-25}) theoretically rely on additional restrictions on the pseudo-gradient operator, such as (partial) strong monotonicity \cite{EC-21,GPS-22} or cocoercivity \cite{ECp-18,YP-17,YP-19,RLZ-24,WWW-25}. These conditions can be restrictive in applications and are often difficult to verify in practice. While some existing works try to relax these additional assumptions by Tseng's forward-backward-forward scheme \cite{Tseng-00}, see e.g. \cite{ECp-18}, the resulting methods still require solving, at each iteration, a large operator equation that couples all players, which is impractical.

Compared with the deterministic literature on monotone GNEPs, results for the GNEP (\ref{sGNEP}) in stochastic and infinite-dimensional settings remain limited. A recent development is the extension of the Moreau-Yosida regularization method to a class of risk-neutral PDE-constrained GNEPs \cite{GHS-23}, which can be viewed as a concrete example of Problem (\ref{sGNEP}). Although this method exhibits impressive convergence properties, as in the deterministic case, it requires solving a sequence of increasingly ill-conditioned risk-neutral NEP subproblems. Moreover, the stochastic setting complicates the computation of the NEP subproblems.

In order to (i) avoid solving difficult NEP subproblems encountered in regularization methods, (ii) overcome the difficulties in extending BR schemes and PHA frameworks in stochastic settings, and (iii) relax the theoretical restrictions of existing splitting methods while keeping the scheme simple, we develop a novel, easy-to-implement single-loop alternating proximal ALM method for the monotone GNEP (\ref{sGNEP}). Our convergence analysis is established without the additional regularity assumptions that are typically required by existing splitting-based approaches. 

The algorithmic design proceeds in two steps. First, we reformulate (\ref{sGNEP}) as an equality-constrained GNEP. Second, we introduce a proximal ALM in which the usual augmented Lagrangian function is replaced by quadratic surrogates. A direct application would still require solving a coupled NEP at each iteration; to eliminate this burden, we incorporate a symmetric Gauss–Seidel (sGS) decomposition, yielding an sGS-based alternating proximal ALM that updates the players' strategies alternately and thus decouples the original NEP subproblems. This decoupling idea is inspired by the sGS techniques for large-scale convex composite conic programs \cite{XDK-19,LDK-17,LXD-21}. In the special case of a single player, our method reduces to a linearized proximal ALM for linearly constrained convex optimization problems.

The main features of the proposed method are summarized below. 
\begin{enumerate}
    \item[1.] By introducing a new quadratic surrogate model and embedding the sGS decomposition, each iteration avoids solving coupled and complicated NEPs, which are typically encountered in regularization methods \cite{MTA-15,MT-13,GHS-23} and ALM-type methods \cite{CD-19,CV-19}.
    \item[2.] The global convergence of our method requires the pseudo-gradient operator to be monotone. Compared with the assumptions imposed on the pseudo-gradient mapping in the existing splitting-based algorithms, such as (partial) strong monotonicity \cite{EC-21,GPS-22} or cocoercivity \cite{ECp-18,YP-17,YP-19,RLZ-24,WWW-25}, our assumption is weaker and easier to verify in applications.
\end{enumerate}
 In addition, each subproblem can be solved inexactly, which can further reduce per-iteration cost in practice. 

The remainder of the paper is organized as follows. In Section~\ref{Sec: Pre}, we introduce the problem formulation, recall the relevant notions of equilibrium and optimality, and state the standing assumptions. In Section \ref{Sec: Alg}, an sGS-based alternating proximal ALM method is derived in detail. Section \ref{Sec: The} is dedicated to analyzing the convergence properties and estimating the non-ergodic convergence rate. In Section \ref{Sec: Pra-ver}, we give a practical sGS-based alternating proximal ALM, which incorporates an adaptive parameter-selection strategy for the construction of quadratic surrogate models. In Section~\ref{Sec: App}, we apply the proposed method to a class of risk-neutral PDE-constrained GNEPs. Some numerical results are given to demonstrate the effectiveness and efficiency of the method. Finally, some concluding remarks are made in Section \ref{Sec: Con}.

\section{Preliminaries}\label{Sec: Pre}
 In this section, we first present a detailed formulation of Problem (\ref{sGNEP}) and review the definition of Nash equilibrium together with optimality conditions. We then show the existence of solutions and list the assumptions used in the subsequent theoretical analysis. Finally, we review some function spaces that are necessary for the risk-neutral PDE-constrained GNEPs considered in Section \ref{Sec: App}.

\subsection{Problem Formulation}
We first recall the definition of equilibrium for Problem (\ref{sGNEP}) \cite{MTA-15,CV-19}. For clarity, some notation is introduced as follows:
\begin{equation}\label{notations-feas}
\begin{aligned}
    &Ax:=\sum_{\nu=1}^NA_\nu x_\nu,\quad 
    \mathcal{X}:=\{x\in X\mid\phi_\nu(x_\nu)<\infty, \text{ for }\nu=1,2,\ldots,N\}.\\
   & \mathcal{F}:=\{x\in \mathcal{X}\mid Ax\in K\},\quad
    \mathcal{F}_\nu(x_{-\nu}):=\{x_\nu\in X_\nu\mid (x_\nu,x_{-\nu})\in\mathcal{F}\},
 \end{aligned}
\end{equation}
where $\mathcal{F}$ represents the feasible set of Problem (\ref{sGNEP}), and $\mathcal{F}_\nu(x_{-\nu})$ is the feasible set of the $\nu$-th player for some given $x_{-\nu}$.

\begin{definition}
    Let $\bar{x}\in\mathcal{F}$ be a feasible point. Then,
    \begin{enumerate}
        \item[(a).] $\bar{x}$ is a generalized Nash equilibrium of the GNEP (\ref{sGNEP}), if for each player $\nu$ and any $y_\nu\in\mathcal{F}_\nu(\bar{x}_{-\nu})$ there holds
    $J_\nu(\bar{x}_\nu,\bar{x}_{-\nu})+\phi_\nu(\bar{x}_\nu)\leq J_\nu(y_\nu,\bar{x}_{-\nu})+\phi_\nu({y}_\nu)$.
      \item[(b).] $\bar{x}$ is a normalized Nash equilibrium of the GNEP (\ref{sGNEP}), if there holds\\
           $\sum_{\nu=1}^N[J_\nu(\bar{x}_\nu,\bar{x}_{-\nu})+\phi_\nu(\bar{x}_\nu)]\leq \sum_{\nu=1}^N[J_\nu(y_\nu,\bar{x}_{-\nu})+\phi_\nu(y_\nu)],\quad \forall y\in \mathcal{F}$.
    \end{enumerate}
\end{definition}
Indeed, each normalized Nash equilibrium is also a generalized Nash equilibrium \cite{CV-19}, and thus our target in this work is to find a normalized Nash equilibrium of (\ref{sGNEP}). To this end, we first analyze the existence of normalized Nash equilibria. Recall that the Nikaido-Isoda function $\Psi(x,y)$ is
\begin{equation}\label{NI-def}
    \Psi(x,y)=\sum_{\nu=1}^N\left(J_\nu({x}_\nu,{x}_{-\nu})+\phi_\nu({x}_\nu)-J_\nu(y_\nu,{x}_{-\nu})-\phi_\nu(y_\nu)\right).
\end{equation}
Notice that $\bar{x}$ is a normalized Nash equilibrium of Problem (\ref{sGNEP}) if and only if $\bar{x}$ is a solution of the optimization problem $\max_{y\in\mathcal{F}} \Psi(\bar{x},y)$. 

Before proceeding with a further statement, we would like to introduce some necessary notations. For $x=(x_1,x_2,\ldots,x_N)\in X$, define the pseudo-gradient operator $P_J: X\to X^*$ and the subdifferential mapping $\partial\phi: X\rightrightarrows X^*$ by
\begin{equation}\label{def-gra}
    P_J(x)=\left(
    \begin{array}{c}
    D_{x_1}J_1(x_1,x_{-1})\\
    D_{x_2}J_2(x_2,x_{-2})\\
    \vdots\\
    D_{x_N}J_N(x_N,x_{-N})
    \end{array}
    \right),\quad \partial\phi(x)=\left(
    \begin{array}{c}
    \partial\phi_1(x_1)\\
    \partial\phi_2(x_2)\\
    \vdots\\
    \partial\phi_N(x_N)
    \end{array}
    \right),
\end{equation}
where $D_{x_\nu}$ represents the Fr\'echet derivative
 with respect to the variable $x_\nu$, and $\partial\phi_\nu$ is the subdifferential of $\phi_\nu$. 
It follows from the literature \cite{FP-07} that the considered GNEP (\ref{sGNEP}) is called monotone if the pseudo-gradient operator $P_J$ is monotone. The rest of this work mainly assumes that the problem is a monotone GNEP. 

\begin{theorem}\label{NE-Existence}
    Suppose that the feasible set $\mathcal{F}$ in \eqref{notations-feas} is nonempty, bounded, and closed in the Hilbert space $X$. Assume that the pseudo-gradient mapping $P_J$ is monotone and that the Nikaido-Isoda function $\Psi(x,y)$ in \eqref{NI-def} is weakly sequentially lower semicontinuous with respect to $x$. Then the set of normalized Nash equilibria of the GNEP \eqref{sGNEP} is nonempty.
\end{theorem}
\begin{proof}
    Note that $\mathcal{F}$ is convex, closed and bounded in the Hilbert space $X$, and thus $\mathcal{F}$ is weakly compact. By the monotonicity of the pseudo-gradient operator, there holds that, for each $\nu$, $J_\nu(\cdot,x_{-\nu})$ is convex for each fixed $x_{-\nu}$. Hence, for each fixed $x\in\mathcal F$,
the Nikaido-Isoda function $\Psi(x,\cdot)$ is concave on $\mathcal F$. The rest of the proof follows from the Ky-Fan Theorem (see Thm 1.1 in \cite{AW-03}). 
\end{proof}
The assumption on weak lower semicontinuity of $\Psi(x,y)$ arises from the Ky-Fan Theorem, and this condition is readily verified in some practical problems \cite{MTA-15,HM-08}; see Section 2.1 of \cite{CV-19} for more discussions.

\begin{definition}\label{def-originalKKT}
    A tuple $(\bar{x},\bar{\lambda})\in X\times Y^*$ is called a variational KKT point of the GNEP (\ref{sGNEP}) if it satisfies
      $0\in P_J(\bar{x})+\partial\phi(\bar{x})+A^*\bar{\lambda}, \bar{\lambda}\in \mathcal{N}_{K}(A\bar{x}),  A\bar{x}\in K$,
    where $\mathcal{N}_{K}$ represents the normal cone of the set $K$, defined by
     $\mathcal{N}_{K}(y):=\{s\in Y^*\mid\langle s,\tilde{y}-y\rangle\leq 0,\quad \text{for all }\tilde{y}\in K\}$.
\end{definition}
The next theorem shows that, under some mild regularity conditions, variational KKT points are equivalent to normalized Nash equilibria.
\begin{theorem}\label{Equivalence-KKT-NE}
    If the GNEP (\ref{sGNEP}) has a monotone pseudo-gradient operator, the following equivalence holds
    \begin{enumerate}
        \item[(a).] If $(\bar{x},\bar{\lambda})\in X\times Y^*$ is a variational KKT point, then $\bar{x}$ is a normalized Nash equilibrium.
        \item[(b).] Conversely, let $\bar{x}$ be a normalized Nash equilibrium. If the linear operator $A\in\mathcal{L}(X,Y)$ has a closed range, and there exists a $\tilde{x}\in\mathcal{X}$ such that $
0\in\operatorname{int}_{Y}(A\tilde{x}-K)$,
then there exists a multiplier $\bar{\lambda}\in Y^*$ such that $(\bar{x},\bar{\lambda})$ is a variational KKT point.
    \end{enumerate}
\end{theorem}
\begin{proof}
    This theorem can be proved in a similar way to Theorem 3.3 in \cite{ECp-18}.
\end{proof}
\subsection{Assumptions}
In this subsection, we specify the problem setting and state the necessary assumptions used in our theoretical analysis. 

First, we follow the assumptions and techniques used in \cite{CV-19} to handle the potential theoretical difficulties brought by the Banach-space setting. Assume that there exists a continuous and dense embedding map $e: Y\hookrightarrow H$ for some Hilbert space $H$, and there is a closed and convex set $\mathcal{K}\subseteq H$ with its pre-image $K=e^{-1}(\mathcal{K})$. Consequently, the joint constraints of Problem (\ref{sGNEP}), i.e. $Ax\in K$, can be interpreted as constraints in the Hilbert space $H$, that is, $(e\circ A)x\in\mathcal{K}$, while the feasible set $\mathcal{F}$ can be equivalently rewritten as
$\mathcal{F}:=\{x\in\mathcal{X}\mid (e\circ A)x\in\mathcal{K}\}$.

For convenience, we define the linear operators $\mathcal{A}: X\to H$ and $\mathcal{A}_\nu: X_\nu\to H$ by
\begin{equation}\label{def-mathcalA}
    \mathcal{A}x:=(e\circ A)x=\sum_{\nu=1}^N\mathcal{A}_{\nu}x_\nu,\quad \mathcal{A}_\nu x_\nu:=(e\circ A_\nu)x_\nu.
\end{equation}
In the remainder of the paper, we design the algorithm and present the theoretical analysis based on the equivalent form of Problem (\ref{sGNEP}). As explained in Remark \ref{remark-embedding}, the multiplier of Problem (\ref{sGNEP}) can be recovered through the embedding map $e$.

For further discussions, we recall some concepts for set-valued operators; see Chapter 4 of \cite{HP-17} for more details.
\begin{definition}
    Let $X$ be a Hilbert space and let $T: X\to 2^{X}$ be a set-valued operator. Then, $T$ is called:
    
(i). Monotone if 
   $\langle u-v,x-y\rangle\geq 0, \forall u\in T(x), v\in T(y)$.
   
(ii). Strongly monotone if there exists a constant 
$\rho>0$ such that
$\langle u-v,x-y\rangle\geq\rho\|x-y\|^2, \forall u\in T(x), v\in T(y)$.
   
(iii). $\alpha$-cocoercive if there exists a constant $\alpha>0$ such that $\langle u-v,x-y\rangle\geq\alpha\|u-v\|^2, \forall u\in T(x), v\in T(y)$,
in particular, when $\alpha=1$, the operator $T$ is called firmly nonexpansive.
\end{definition}

Then, we recall the concept of quasi-Fej\'er monotonicity for a sequence.
\begin{theorem}[Quasi-Fej\'er monotonicity]\label{thw}
Let $\{a_k\}$ be a sequence of nonnegative real numbers, for all $k$ there holds $a_{k+1} \le a_k + \varepsilon_k$, 
where $\{\varepsilon_k\}$ is a nonnegative and summable sequence of real numbers. Then, $\{a_k\}$ has a finite limit.
\end{theorem}
The next theorem is a standard consequence on weak convergence.
\begin{theorem}\label{th: weak-converge}
Let $\{u^k\}$ be a bounded sequence in the Hilbert space $U$ and suppose that
$\lim_{k\to\infty} \|u^k-u^\infty\|$ exists whenever $u^\infty$ is a weak limit point of $\{u^k\}$.
Then, the sequence $\{u^k\}$ is weakly convergent.
\end{theorem}

Finally, we summarize the assumptions required for the subsequent analysis.
\begin{assumption}\label{ass1}
    \begin{enumerate}
     \item[(i).] For each $\nu=1,2,\ldots,N$, the space $X_\nu$ is a Hilbert space, and thus $X=X_1\times X_2\times\cdots\times X_N$ is also a Hilbert space. Moreover, there exists a normalized Nash equilibrium for Problem (\ref{sGNEP}).
        \item[(ii).] For each $\nu=1,2,\ldots,N$, the function $\phi_\nu:  X_\nu\to\mathbb{R}\cup\{+\infty\}$ is proper, convex, and lower semicontinuous, the functional $J_\nu(x)$ is continuously Fr\'echet differentiable with respect to $x=(x_{\nu},x_{-\nu})$.
        \item[(iii).] For each $\nu=1,2,\ldots,N$, the linear operator $A_\nu: X_\nu\to Y$ is bounded, and the range of the operator $A$ is closed. 
       \item[(iv).] There exists a point $\tilde{x}\in\mathcal{X}$ ($\mathcal{X}$ is defined in (\ref{notations-feas})), such that $0\in \operatorname{int}_{Y}(A\tilde{x}-K)$.
       \item[(v).] The pseudo-gradient operator $P_J$ defined by (\ref{def-gra}) is monotone and globally Lipschitz continuous with constant $L$ in $X$.
       \item[(vi).] There exists a continuous and dense embedding $e: Y\hookrightarrow H$ with $H$ being a Hilbert space, and a closed and convex set $\mathcal{K}\subseteq H$ such that $K=e^{-1}(\mathcal{K})$. 
    \end{enumerate}
\end{assumption}


\begin{remark}
In the remainder of this paper, we assume the existence of a normalized Nash equilibrium of (\ref{sGNEP}). A sufficient condition is provided in Theorem~\ref{NE-Existence}. Moreover, Theorem~\ref{Equivalence-KKT-NE} shows that, under Assumption~\ref{ass1}, finding a normalized Nash equilibrium is equivalent to computing a variational KKT point.
\end{remark}

\begin{remark}
Unlike many existing splitting-based methods, which typically impose stronger regularity conditions on the corresponding pseudo-gradient operators such as (partial) strong monotonicity \cite{EC-21,GPS-22} or cocoercivity \cite{ECp-18,YP-17,YP-19,RLZ-24,WWW-25}, our analysis relies solely on the monotonicity assumption given by Assumption \ref{ass1} (v).
\end{remark}

\subsection{Preliminaries for Applications}
To apply the designed algorithm to a class of risk-neutral PDE-constrained GNEPs in Section \ref{Sec: App}, we shall introduce the corresponding concepts and define some necessary function spaces.

The triple $(\Omega,\mathcal{E},\mathbb{P})$ represents a complete probability space, where $\Omega$ is the sample space, $\mathcal{E}$ is the Borel $\sigma$-algebra of $\Omega$ for a fixed topology on $\Omega$, and $\mathbb{P}: \mathcal{E}\to [0,1]$ is a probability measure. Let $V$ be a real Banach space endowed with the norm $\|\cdot\|_V$. For a random element $\xi: \Omega\to V$, we define the expectation with respect to $\mathbb{P}$ by $\mathbb{E}_{\mathbb{P}}[\xi]:=\int_{\Omega}\xi(\omega)d\mathbb{P}$.

For a given $1\leq p\leq\infty$ and a complete probability space $(\Omega,\mathcal{E},\mathbb{P})$, the Bochner space $L^p(\Omega;V)$ is defined by\\
    $L^p(\Omega;V)=\{u:\Omega\to V\mid u \text{ is strongly } \mathcal{E}-\text{measurable and } \|u\|_{L^p(\Omega;V)}<\infty\}$, 
where the norm is defined as $\|u\|_{L^p(\Omega;V)}:=\left(\int_\Omega\|u(\omega)\|^p_Vd\mathbb{P}(\omega)\right)^\frac{1}{p}$ for $1\leq p<\infty$, and $\|u\|_{L^p(\Omega;V)}:=\mathop{\text{ess}\sup}_{\omega\in\Omega}\|u(\omega)\|_V$ for $p=\infty$.

Furthermore, for $1\leq p,q<\infty$ such that $\frac{1}{p}+\frac{1}{q}=1$, there holds $(L^p(\Omega;V))^*\cong L^q(\Omega;V^*)$. Moreover, if $V$ is a reflexive Banach space and $1<p<\infty$, the Bochner space $L^p(\Omega;V)$ is reflexive and the duality pair satisfies
$\langle u,v\rangle=\mathbb{E}_{\mathbb{P}}[\langle u(\omega),v(\omega) \rangle_{V^*,V}]$.

For the case that $p=2$ and $V$ is a Hilbert space, the Bochner space $L^2(\Omega;V)$ is a Hilbert space endowed with inner product $\langle u,v\rangle:=\mathbb{E}_{\mathbb{P}}[\langle u(\omega),v(\omega) \rangle_{V}]$. For further properties of Bochner spaces, we refer to Chapter III of \cite{HPh-74}.

\section{An sGS-based Alternating Proximal ALM (sGS-APALM) for Problem (\ref{sGNEP})}\label{Sec: Alg}
In this section, we develop an sGS-based alternating proximal ALM (sGS-APALM) method for Problem (\ref{sGNEP}). To this end, we first rewrite the original problem in the following equivalent form:
\begin{equation}\label{sGNEP-Eq}\tag{Q}
    \begin{aligned}
        \min_{(x_\nu,z_1,z_{2,\nu})}\quad& J_\nu(x_\nu,x_{-\nu})+I_{\mathcal{K}}(z_1)+\phi_\nu(z_{2,\nu})\\
        \mbox{s.t.}\quad&\mathcal{A}x-z_1=0,\quad x_\nu-z_{2,\nu}=0,
    \end{aligned}
\end{equation}
where the operator $\mathcal{A}$ is defined by (\ref{def-mathcalA}). From Assumption~\ref{ass1}(vi), the linear constraints are formulated in a Hilbert space. Consequently, the equivalent reformulation is also posed in a Hilbert space; we refer to this reformulation as Problem~\eqref{sGNEP-Eq}. For simplicity, let us denote $z_2=(z_{2,1},z_{2,2},\ldots,z_{2,N})^\top$, and define
\begin{equation}\label{def-Bpsi}
    \begin{aligned}
       \mathcal{B}(z,x):=\left(\begin{array}{c}
        \mathcal{A}x-z_1 \\
        x-z_2
        \end{array}
       \right),\quad z=(z_1,z_2)^\top,\quad\psi(z):=I_{\mathcal{K}}(z_1)+\sum_{\nu=1}^N\phi_\nu(z_{2,\nu}).
    \end{aligned}
\end{equation}

\begin{remark} 
Note that Problem (\ref{sGNEP-Eq}) is posed in a Hilbert space. By the Riesz representation theorem, the dual space is isomorphic to the space itself; hence, without loss of generality, we carry out the subsequent analysis in the primal Hilbert space.
\end{remark}

\begin{definition}\label{def-VIKKT}
    A tuple $(\bar{x},\bar{z},\bar{\lambda})\in X\times (H\times X)\times (H\times X)$ is called a variational KKT point of the GNEP (\ref{sGNEP-Eq}) if it satisfies $0\in\left(\begin{array}{c}
            \partial\psi(\bar{z})\\ 
            P_J(\bar{x})
    \end{array}\right)+\mathcal{B}^*\bar{\lambda}$ and $0=\mathcal{B}(\bar{z},\bar{x})$,
    where $\partial\psi(\bar{z})$ denotes the subdifferential of $\psi(\cdot)$ at $\bar{z}$.
\end{definition}
Before presenting the designed algorithm, we define the function $\hat{J}_\nu(y_\nu;x,\tilde{x})$ ($\nu=1,2,\ldots,N$) as follows, which serves as a quadratic surrogate model of $J_\nu(y_\nu,x_{-\nu})$ at the given point $x=(x_\nu,x_{-\nu})$ and $\tilde{x}=(\tilde{x}_\nu,\tilde{x}_{-\nu})$:
\begin{equation}\label{quad-approx}
\hat{J}_\nu(y_\nu;x,\tilde{x}):=J_\nu(x)+\langle 2D_{x_\nu}J_\nu(x)-D_{x_\nu}J_\nu(\tilde{x}),y_\nu-x_\nu \rangle+\frac{1}{2}\|y_\nu-x_\nu\|^2_{\Sigma_\nu},
\end{equation}
where $\Sigma_\nu$ is self-adjoint and strongly monotone in $X_\nu$, and $\|y_\nu-x_\nu\|^2_{\Sigma_\nu}:=\langle \Sigma_\nu(y_\nu-x_\nu),y_\nu-x_\nu \rangle_{X_\nu}$.

Given iterates $x^k$ and $x^{k-1}$, we define the $\nu$-th quadratic surrogate model of the corresponding augmented Lagrangian function, as follows:
\begin{equation}\label{def-linearizedALM}
\hat{L}^\nu_\sigma(x,z;x^k,x^{k-1},\lambda):=\hat{J}_\nu(x_\nu;x^k,x^{k-1})+\psi(z)+\langle\lambda,\mathcal{B}(z,x)\rangle+\frac{\sigma}{2}\|\mathcal{B}(z,x)\|^2.
\end{equation}
Then, we can directly extend the proximal ALM to solve (\ref{sGNEP-Eq}). The method is iteratively defined by first solving the following NEP
\begin{equation}\label{subNEP}
		\min_{(x_\nu,z_1,z_{2,\nu})} \hat{L}^{\nu}_{\sigma}(x,z;x^k,x^{k-1},\lambda^k),\quad \text{for $\nu=1,2,\ldots,N$},
	 \end{equation}
 and then update the dual multiplier by $\lambda^{k+1}=\lambda^k+\tau\sigma\mathcal{B}(z^{k+1},x^{k+1})$.    
Note that each iteration requires solving the coupled NEP (\ref{subNEP}), which can be computationally expensive, especially when the number of players and the dimension of each decision variable are large. To avoid this difficult subproblem, we employ the symmetric Gauss-Seidel (sGS) decomposition technique to decouple it. Define $x_{\leq\nu-1}:=(x_1,x_2,\ldots,x_{\nu-1})^{\top}$ and $x_{\geq\nu+1}:=(x_{\nu+1},x_{\nu+2},\ldots,x_{N})^{\top}$. The proposed sGS-based alternating proximal ALM (sGS-APALM) is obtained in Algorithm \ref{ALM}.

In each iteration of the sGS-APALM (Algorithm \ref{ALM}), the coupled NEP subproblem (\ref{subNEP}) is replaced by $2N+1$ single-objective optimization problems. Moreover, we update each player's strategy alternately. In particular, the subproblems in both backward and forward sweeping procedures are unconstrained quadratic programs. Consequently, Algorithm \ref{ALM} can be implemented efficiently. This decoupling technique is mainly motivated by the sGS decomposition technique for large-scale convex composite conic programming problems; see \cite{XDK-19,LDK-17,LXD-21} for details. Besides, it is worth noting that in the special case of only one player, sGS-APALM reduces to a linearized proximal ALM.

\begin{remark}\label{remark-embedding}
    Comparing the KKT conditions of the original problem (Definition \ref{def-originalKKT}) with those of the equivalent problem (Definition \ref{def-VIKKT}), we find that the primal solution $x$ coincides, whereas the multiplier for the original problem can be obtained by $e^*(\lambda_1)$ with $\lambda=(\lambda_1,\lambda_2)$ being the optimal multiplier of the equivalent problem.    
\end{remark}
\begin{remark}
Unlike standard quadratic surrogates (e.g., the majorization models in \cite{LXD-21}) that rely only on $D_{x_\nu}J_\nu(x^k)$, our surrogate \eqref{quad-approx} is more involved. However, this construction allows us to prove convergence of sGS-APALM (Algorithm \ref{ALM}) without imposing stronger regularity conditions on the pseudo-gradient that are commonly required by existing splitting-based methods. Besides, our idea is somewhat similar to the reflected forward-backward schemes for monotone inclusions \cite{BNZ-25,YMT-20}, but here it is for GNEPs, which have a fundamentally different structure. 
\end{remark}

\begin{algorithm}[ht]
		\caption{An sGS-based alternating proximal ALM (sGS-APALM) for GNEP (\ref{sGNEP-Eq})}
		\hspace*{0.02in} {\bf Step 1:} Choose the penalty parameter $\sigma>0$, and the dual step-size $\tau\in (0,2)$. Initialize $\lambda^0\in H\times X$, $z^0\in H\times X$, and $x^0\in X$, set $k=0$, and let $x^{-1}=x^0$. \\
		\hspace*{0.02in} {\bf Step 2a:} Backward Sweeping: For $\nu=N,N-1,\ldots,1$, update $\bar{x}_{\nu}^{k+1}$ by inexactly solving 
		\begin{equation}\label{back}
		\min_{x_{\nu}} \hat{L}^{\nu}_{\sigma}(x^k_{\leq\nu-1},x_{\nu},\bar{x}^{k+1}_{\geq\nu+1},z^k;x^k,x^{k-1},\lambda^k).
		\end{equation}
  \hspace*{0.02in} {\bf Step 2b:} Update $z^{k+1}$ by solving 
		\begin{equation}\label{up-z}
		\min_{z} \psi(z)+\langle\lambda^k,\mathcal{B}(z,\bar{x}^{k+1})\rangle+\frac{\sigma}{2}\|\mathcal{B}(z,\bar{x}^{k+1})\|^2.
		\end{equation}
  \hspace*{0.02in} {\bf Step 2c:} Forward Sweeping: For $\nu=1,2,\ldots,N$, update $x_{\nu}^{k+1}$ by inexactly solving 
		\begin{equation}\label{for}
		\min_{x_{\nu}} \hat{L}^{\nu}_{\sigma} (x^{k+1}_{\leq\nu-1},x_{\nu},\bar{x}^{k+1}_{\geq\nu+1},z^{k+1};x^k,x^{k-1},\lambda^k).
		\end{equation}
		\hspace*{0.02in} {\bf Step 3:} Update multiplier $\lambda^{k+1}$ by
		\begin{equation}\label{up-lam}
	\lambda^{k+1}=\lambda^k+\tau\sigma\mathcal{B}(z^{k+1},x^{k+1}).
		\end{equation}
		\hspace*{0.02in} {\bf Step 4:}
		 Set $k=k+1$ and return to \textbf{Step 2}.\\
   \label{ALM}
	\end{algorithm}

\section{Convergence Analysis}\label{Sec: The}
This section establishes convergence and estimates the non-ergodic convergence rate for sGS-APALM (Algorithm \ref{ALM}). It is worth noting that our analysis does not require additional regularity conditions which are typically imposed in existing splitting-based methods, such as (partial) strong monotonicity \cite{EC-21,GPS-22} or cocoercivity \cite{ECp-18,YP-17,YP-19,RLZ-24,WWW-25}.

\subsection{Preliminaries}
Before presenting the detailed analysis, we define the set-valued operator $F: (H\times X)\times X\times (H\times X)\rightrightarrows (H\times X)\times X\times  (H\times X)$, which is induced by the KKT conditions in Definition \ref{def-VIKKT}.
\begin{equation}\label{def-F}
F(z,x,\lambda):=\left(
    \begin{array}{cc}
          \left(\begin{array}{cc}
            \partial\psi(\cdot) & \\
             & P_J(\cdot)
        \end{array}\right) &\mathcal{B}^*\\
      -\mathcal{B} &
       \end{array}\right)\left(\begin{array}{c}
            z \\
            x\\
            \lambda
       \end{array}\right)=\left(
    \begin{array}{cc}
          \left(\begin{array}{c}
             \partial\psi(z) \\
             P_J(x)
        \end{array}\right) +\mathcal{B}^*\lambda\\
      -\mathcal{B}(z,x)
       \end{array}\right),
\end{equation}
where $\lambda=(\lambda_1,\lambda_2)^\top$, and $\lambda_1,\lambda_2$ are the multipliers that correspond to the constraints $\mathcal{A}x=z_1$ and $x=z_2$, respectively. Moreover, we denote $\lambda_2=(\lambda_{2,1},\lambda_{2,2},\ldots,\lambda_{2,N})^\top$, where $\lambda_{2,i}$ is the multiplier associated with the constraint $x_i=z_{2,i}$. Then, the KKT conditions of Problem (\ref{sGNEP-Eq}) can be compactly rewritten as $0\in F(z,x,\lambda)$.
\begin{lemma}\label{operatorF-prop}
    The operator $F$ is maximal monotone in $(H\times X)\times X\times (H\times X)$, and there holds
    \begin{equation}\label{ineq1}
    \begin{aligned}
        &\langle f_1-f_2,(z_1,x_1,\lambda_1)-(z_2,x_2,\lambda_2)\rangle\geq \langle P_J(x_1)-P_J(x_2),x_1-x_2 \rangle\geq 0,\\
         &\forall f_1\in F(z_1,x_1,\lambda_1), f_2\in F(z_2,x_2,\lambda_2).
         \end{aligned}
    \end{equation}
\end{lemma}
\begin{proof}
 The inequality (\ref{ineq1}) follows directly from algebraic manipulations and the monotonicity of pseudo-gradient operator $P_J(\cdot)$. Since $\partial\psi$ is maximal monotone, $P_J(\cdot)$ is a single-valued continuous monotone operator, and the remaining linear part is bounded and skew-adjoint, the maximal monotonicity of $F$ holds from the standard sum theorem for maximal monotone operators.
\end{proof}
Next, we give the optimality conditions for (\ref{back}), (\ref{up-z}), and (\ref{for}).
\begin{lemma}\label{Lemma: sub}
    The update schemes in (\ref{back}), (\ref{up-z}), and (\ref{for}) are equivalent to the following systems: 
 \begin{equation}\label{back-KKT}
   \left\{
    \begin{aligned}
        \delta_N&=(2D_{x_N}J_{N}(x^k)-D_{x_N}J_{N}(x^{k-1}))+\mathcal{A}_N^*\lambda_1^k\\
        &+\sigma\mathcal{A}_N^*(-z_1^k+\mathcal{A}_1x_1^k+\cdots+\mathcal{A}_{N-1}x_{N-1}^{k}+\mathcal{A}_N\bar{x}_N^{k+1})\\
        &+\Sigma_N(\bar{x}^{k+1}_N-x_N^k)+\lambda_{2,N}^k+\sigma(\bar{x}^{k+1}_N-z^k_{2,N})\\
        \delta_{N-1}&=(2D_{x_{N-1}}J_{N-1}(x^k)-D_{x_{N-1}}J_{N-1}(x^{k-1}))+\mathcal{A}_{N-1}^*\lambda_1^k\\
        &+\sigma\mathcal{A}_{N-1}^*(-z_1^k+\mathcal{A}_1x_1^k+\cdots+\mathcal{A}_{N-1}\bar{x}_{N-1}^{k+1}+\mathcal{A}_N\bar{x}_N^{k+1})\\
        &+\Sigma_{N-1}(\bar{x}^{k+1}_{N-1}-x_{N-1}^k)+\lambda_{2,N-1}^k+\sigma(\bar{x}^{k+1}_{N-1}-z^k_{2,N-1})\\
        &\vdots\\
        \delta_1&=(2D_{x_{1}}J_{1}(x^k)-D_{x_{1}}J_{1}(x^{k-1}))+\mathcal{A}_{1}^*\lambda_1^k\\
        &+\sigma\mathcal{A}_{1}^*(-z_1^k+\mathcal{A}_1\bar{x}_1^{k+1}+\cdots+\mathcal{A}_{N-1}\bar{x}_{N-1}^{k+1}+\mathcal{A}_N\bar{x}_N^{k+1})\\
        &+\Sigma_{1}(\bar{x}^{k+1}_{1}-x_{1}^k)+\lambda_{2,1}^k+\sigma(\bar{x}^{k+1}_{1}-z^k_{2,1}),\\
    \end{aligned}
    \right.
    \end{equation}
    \begin{equation}\label{up-zKKT}
        0\in \partial\psi(z^{k+1})-\lambda^{k}-\sigma\left(\begin{array}{c}-z_1^{k+1}+\mathcal{A}_1\bar{x}_1^{k+1}+\cdots+\mathcal{A}_{N-1}\bar{x}_{N-1}^{k+1}+\mathcal{A}_N\bar{x}_N^{k+1}\\
        -z_2^{k+1}+\bar{x}^{k+1}
        \end{array}
        \right),
    \end{equation}
\begin{equation}\label{forKKT}
   \left\{
    \begin{aligned}
    \delta^\prime_1&=(2D_{x_{1}}J_{1}(x^k)-D_{x_{1}}J_{1}(x^{k-1}))+\mathcal{A}_{1}^*\lambda_1^k\\
        &+\sigma\mathcal{A}_{1}^*(-z_1^{k+1}+\mathcal{A}_1{x}_1^{k+1}+\mathcal{A}_{2}\bar{x}_{2}^{k+1}+\cdots+\mathcal{A}_N\bar{x}_N^{k+1})\\
        &+\Sigma_{1}({x}^{k+1}_{1}-x_{1}^k)+\lambda_{2,1}^k+\sigma(x^{k+1}_{1}-z^{k+1}_{2,1})\\
    \delta^\prime_2&=(2D_{x_{2}}J_{2}(x^k)-D_{x_{2}}J_{2}(x^{k-1}))+\mathcal{A}_{2}^*\lambda_1^k\\
        &+\sigma\mathcal{A}_{2}^*(-z_1^{k+1}+\mathcal{A}_1{x}_1^{k+1}+\mathcal{A}_2{x}_2^{k+1}+\cdots+\mathcal{A}_N\bar{x}_N^{k+1})\\
        &+\Sigma_{2}({x}^{k+1}_{2}-x_{2}^k)+\lambda_{2,2}^k+\sigma(x^{k+1}_{2}-z^{k+1}_{2,2}),\\
    &\vdots\\
        \delta^\prime_N&=(2D_{x_{N}}J_{N}(x^k)-D_{x_{N}}J_{N}(x^{k-1}))+\mathcal{A}_{N}^*\lambda_1^k\\
        &+\sigma\mathcal{A}_{N}^*(-z_1^{k+1}+\mathcal{A}_1{x}_1^{k+1}+\mathcal{A}_2{x}_2^{k+1}+\cdots+\mathcal{A}_N{x}_N^{k+1})\\
        &+\Sigma_{N}({x}^{k+1}_{N}-x_{N}^k)+\lambda_{2,N}^k+\sigma(x^{k+1}_{N}-z^{k+1}_{2,N}),\\
    \end{aligned}
    \right.
    \end{equation}
    where $\{\delta_i\}_{i=1}^N$ and $\{\delta_i^\prime\}_{i=1}^N$ are the residuals caused by solving each subproblem inexactly, with $\delta_i,\delta^\prime_i\in X_i$.
\end{lemma}
\begin{proof}
    By Assumption \ref{ass1} (ii) and the definition of $\hat{L}_\sigma(x,z;x^k,x^{k-1},\lambda^k)$ given by (\ref{def-linearizedALM}), each subproblem of (\ref{back}), (\ref{up-z}) and (\ref{for}) is convex.  The rest of the proof follows from Fermat's rule.
\end{proof}
The next theorem rewrites the systems (\ref{back-KKT}), (\ref{up-zKKT}), and (\ref{forKKT}) in a compact form. For later use, we introduce the operators $\mathcal{M},\mathcal{D},\mathcal{U}$, where $\mathcal{D}$ and $\mathcal{U}$ represent the diagonal and upper triangular parts of $\mathcal{M}$, respectively,
\begin{equation}\label{operator-def}
\begin{aligned}
    &\mathcal{M}:=\left(
    \begin{array}{ccc}
        \sigma\rm{I} &  &-\sigma\mathcal{A}\\
         & \sigma\rm{I} &-\sigma \rm{I}\\
         -\sigma\mathcal{A}^* &-\sigma\rm{I} &(\sigma\rm{I}+\sigma\mathcal{A}^*\mathcal{A}+\Sigma) 
    \end{array}
    \right),\\
    &\mathcal{D}:=\operatorname{diag}(\sigma\rm{I},\sigma\rm{I},\sigma\rm{I}+\sigma\mathcal{A}_1^*\mathcal{A}_1+\Sigma_1,\ldots,\sigma{\rm{I}}+\sigma\mathcal{A}_N^*\mathcal{A}_N+\Sigma_N),\\
&\mathcal{U}:=\left(
\begin{array}{cccccc}
  0& 0& 0& 0&\cdots &-\sigma\mathcal{A}\\
   0 & 0& 0& 0&\cdots &-\sigma\rm{I}\\ 
    0 & 0&0& \sigma\mathcal{A}_1^*\mathcal{A}_2& \cdots &\sigma\mathcal{A}_1^*\mathcal{A}_N\\
     0 & 0 & 0 &0 & \cdots&\sigma\mathcal{A}_2^*\mathcal{A}_N\\
    \vdots & \vdots&  \vdots& \vdots & \ddots&\sigma\mathcal{A}_{N-1}^*\mathcal{A}_N\\
    0& 0& 0& 0& \cdots& 0
\end{array}\right),\\
\end{aligned}
\end{equation}
where $\Sigma$ is defined as $\Sigma:=\operatorname{diag}(\Sigma_1,\Sigma_2,\ldots,\Sigma_N)$.
\begin{lemma}
    The linear operators $\mathcal{M},\mathcal{D},\mathcal{U}$ defined in (\ref{operator-def}), which map from $ H\times X\times X$ to itself, are bounded. Moreover, the operator $\mathcal{M}$ is self-adjoint, and the operator $\mathcal{D}$ is self-adjoint and strongly monotone.
\end{lemma}
\begin{proof}
    The conclusions directly follow from the definitions of $\mathcal{M}$, $\mathcal{D}$, and $\mathcal{U}$, together with Assumption \ref{ass1} (iii).
\end{proof}
\begin{theorem}\label{iteration-VI}
    The iterate $(x^{k+1},z^{k+1})$ obtained by (\ref{back}), (\ref{up-z}) and (\ref{for}) satisfies the following variational system:
    \begin{equation}\label{sGS-sys1}
    \begin{aligned}
        d_k\in&\left(
        \begin{array}{c}
             \partial\psi(z^{k+1})  \\
             2P_J(x^k)-P_J(x^{k-1})
        \end{array}
\right)+\left(\mathcal{T}+\widetilde{\Sigma}\right)\left(\begin{array}{c}
             z^{k+1}-z^k  \\
             x^{k+1}-x^k 
        \end{array}\right)\\
        +&\mathcal{B}^*\lambda^k+\sigma\mathcal{B}^*\mathcal{B}(z^{k+1},x^{k+1}),
        \end{aligned}
    \end{equation}
    where the operator $\mathcal{T}:=\mathcal{U}\mathcal{D}^{-1}\mathcal{U}^*$, the residual $d_k$ and the operator $\widetilde{\Sigma}$ are defined as
\begin{equation}\label{sys1}
    d_k:=\delta^\prime+\mathcal{U}\mathcal{D}^{-1}(\delta^\prime-\delta),\quad \widetilde{\Sigma}:=\left(\begin{array}{cc}
            0 &  \\
             & \Sigma
        \end{array}\right),
\end{equation}
with $\delta:=(0,0,\delta_1,\delta_2,\ldots,\delta_N)$ and $\delta^\prime:=(0,0,\delta_1^\prime,\delta_2^\prime,\ldots,\delta_N^\prime)$ defined in Lemma \ref{Lemma: sub}.
\end{theorem}
\begin{proof}
    The proof follows an argument similar to Theorem 1 in \cite{XDK-19}.
\end{proof}
\begin{remark}
    Comparing the equivalent form (\ref{sGS-sys1}) with the optimality conditions of the NEP subproblem (\ref{subNEP}), we note the additional term $\mathcal{T}\left(\begin{array}{c}
        z^{k+1}-z^k \\
        x^{k+1}-x^k
    \end{array}\right)$. This term is induced by the symmetric Gauss-Seidel (sGS) decomposition technique. Due to this term, we update each player's decision variable alternately, and thus avoid solving the original coupled NEP subproblem. 
\end{remark}
\subsection{Global Convergence}
In this section, we establish the global convergence property of Algorithm \ref{ALM}. 
For simplicity, let $v^k=(z^k,x^k)$. Combining the definition of $F(\cdot)$ in (\ref{def-F}) with Theorem \ref{iteration-VI} and the dual update scheme (\ref{up-lam}), there holds
\begin{equation}\label{KKT-residual}
    R_k\in F(z^{k+1},x^{k+1},\lambda^{k+1}),
\end{equation}
where the residual $R_k$ is given by
\begin{equation}\label{residual-def}
\begin{aligned}
    R_k&=\left(
    \begin{array}{c}
    d_k+\left(\mathcal{T}+\widetilde{\Sigma}\right)(v^k-v^{k+1})-(1-\frac{1}{\tau})\mathcal{B}^*(\lambda^{k}-\lambda^{k+1})\\
        \frac{\lambda^k-\lambda^{k+1}}{\tau\sigma}
        \end{array}
        \right)\\
    &+\left(\begin{array}{c}
             0 \\
            P_J(x^{k+1})-2P_J(x^k)+P_J(x^{k-1})\\
            0
        \end{array}\right)
        \end{aligned}
\end{equation}
The next Theorem establishes the relationship between the iterate and the KKT point.
\begin{theorem}\label{Th-ineq}
    Suppose that Assumption \ref{ass1} holds, the linear operators $\Sigma_\nu$, $(\nu=1,2,\ldots,N)$ satisfy $\Sigma_\nu\succ 2L{\rm{I}}$ with $L$ being the global Lipschitz constant defined in Assumption \ref{ass1} (v), and let $\{(v^k,\lambda^k)\}$ be the sequence generated by Algorithm \ref{ALM} with $v^k:=(z^k,x^k)$. For any given variational KKT point $(v^*,\lambda^*)$ ($v^*:=(z^*,x^*)$) of the monotone GNEP (\ref{sGNEP-Eq}), let us define
    \begin{equation}\label{def-Phifun}
    \begin{aligned}
        \Phi(v^k,v^{k-1},\lambda^k):=&\|v^{k}-v^*\|^2_{\mathcal{T}+\widetilde{\Sigma}}+\frac{1}{\tau\sigma}\|\lambda^k-\lambda^*\|^2\\
        &-2\langle P_J(x^{k})-P_J(x^{k-1}),x^{k}-x^*\rangle+L\|x^k-x^{k-1}\|^2.
        \end{aligned}
    \end{equation}
     Then, there holds
    \begin{enumerate}
    \item[(a).] each iterate satisfies
   \begin{equation*}
       \begin{aligned}
       &\frac{1}{2}(\Phi(v^k,v^{k-1},\lambda^k)-\Phi(v^{k+1},v^k,\lambda^{k+1}))
       +\langle d_k,v^{k+1}-v^*\rangle\\
       \geq&\frac{1}{2}\|v^k-v^{k+1}\|^2_{\mathcal{T}+\widetilde{\Sigma}}-L\|x^{k+1}-x^k\|^2
       +\left(\frac{1}{\tau}-\frac{1}{2}\right)\frac{1}{\tau\sigma}\|\lambda^k-\lambda^{k+1}\|^2.
       \end{aligned}
   \end{equation*}
   \item[(b).] the sequence $\{\Phi(v^k,v^{k-1},\lambda^k)\}$ is nonnegative, and there holds
\begin{equation*}
    \Phi(v^k,v^{k-1},\lambda^k)\geq \|v^k-v^*\|^2_{\mathcal{T}+\frac{\widetilde{\Sigma}}{2}}+\frac{1}{\tau\sigma}\|\lambda^k-\lambda^*\|^2.
\end{equation*}
   \end{enumerate}
\end{theorem}
\begin{proof}
(a). Combining Lemma \ref{operatorF-prop} with (\ref{KKT-residual}) and (\ref{residual-def}), there holds
\begin{equation}\label{ineq2}
\begin{aligned}
    &\langle d_k,v^{k+1}-v^*\rangle+\langle(\mathcal{T}+\widetilde{\Sigma})(v^k-v^{k+1}),v^{k+1}-v^*\rangle\\
    +&\langle P_J(x^{k+1})
    -2P_J(x^{k})+P_J(x^{k-1}),x^{k+1}-x^*\rangle\\
    -&\left(1-\frac{1}{\tau}\right)\langle\mathcal{B}^*(\lambda^{k}-\lambda^{k+1}),v^{k+1}-v^*\rangle+\frac{1}{\tau\sigma}\langle\lambda^{k}-\lambda^{k+1},\lambda^{k+1}-\lambda^*\rangle
    \geq 0
\end{aligned}
\end{equation}
From (\ref{up-lam}) and the feasibility of $v^*$, we have
    $\langle\mathcal{B}^*(\lambda^{k+1}-\lambda^k),v^{k+1}-v^*\rangle=\frac{1}{\tau\sigma}\|\lambda^{k+1}-\lambda^k\|^2$.
Besides, there also hold
        $\langle\lambda^{k}-\lambda^{k+1},\lambda^{k+1}-\lambda^*\rangle=\frac{1}{2}(\|\lambda^k-\lambda^*\|^2-\|\lambda^{k+1}-\lambda^*\|^2-\|\lambda^{k+1}-\lambda^k\|^2)$ and
        $\langle(\mathcal{T}+\widetilde{\Sigma})(v^k-v^{k+1}),v^{k+1}-v^*\rangle=\frac{1}{2}(\|v^k-v^*\|^2_{\mathcal{T}+\widetilde{\Sigma}}-\|v^{k+1}-v^*\|^2_{\mathcal{T}+\widetilde{\Sigma}}-\|v^{k+1}-v^k\|^2_{\mathcal{T}+\widetilde{\Sigma}})$.
Combining these with (\ref{ineq2}) and after simple manipulations, we have
\begin{equation}\label{ineq3}
       \begin{aligned}
       &\frac{1}{2}\left(\|v^{k}-v^*\|^2_{\mathcal{T}+\widetilde{\Sigma}}+\frac{1}{\tau\sigma}\|\lambda^k-\lambda^*\|^2\right)-\frac{1}{2}\left(\|v^{k+1}-v^*\|^2_{\mathcal{T}+\widetilde{\Sigma}}+\frac{1}{\tau\sigma}\|\lambda^{k+1}-\lambda^*\|^2\right)\\
       +&\langle d_k,v^{k+1}-v^*\rangle
       \geq\frac{1}{2}\|v^k-v^{k+1}\|^2_{\mathcal{T}+\widetilde{\Sigma}}+\left(\frac{1}{\tau}-\frac{1}{2}\right)\frac{1}{\tau\sigma}\|\lambda^k-\lambda^{k+1}\|^2\\
       -&\langle P_J(x^{k+1})-P_J(x^{k}),x^{k+1}-x^*\rangle
       +\langle P_J(x^{k})-P_J(x^{k-1}),x^{k+1}-x^*\rangle.
       \end{aligned}
   \end{equation}

 Then, we shall analyze the last term $\langle P_J(x^{k})-P_J(x^{k-1}),x^{k+1}-x^*\rangle$.
Note that
\begin{equation}\label{ineq4}
\begin{aligned}
    &\langle P_J(x^{k})-P_J(x^{k-1}),x^{k+1}-x^*\rangle\\
    =&\langle P_J(x^{k})-P_J(x^{k-1}),x^{k}-x^*\rangle+\langle P_J(x^{k})-P_J(x^{k-1}),x^{k+1}-x^k\rangle\\
   \geq&\langle P_J(x^{k})-P_J(x^{k-1}),x^{k}-x^*\rangle-L\|x^{k}-x^{k-1}\|\|x^{k+1}-x^k\|\\
   \geq&\langle P_J(x^{k})-P_J(x^{k-1}),x^{k}-x^*\rangle-\frac{L}{2}\|x^{k}-x^{k-1}\|^2-\frac{L}{2}\|x^{k+1}-x^k\|^2.
    \end{aligned}
\end{equation}
where the inequalities hold from the Cauchy-Schwarz inequality. 
Combining (\ref{ineq3}) with (\ref{ineq4}), the conclusion of (a) holds.

(b). From Assumption \ref{ass1} (v) and the Cauchy-Schwarz inequality, we obtain $\Phi(v^k,v^{k-1},\lambda^k)\geq\|v^{k}-v^*\|^2_{\mathcal{T}+\widetilde{\Sigma}}+\frac{1}{\tau\sigma}\|\lambda^k-\lambda^*\|^2-2L\|x^k-x^{k-1}\|\|x^{k}-x^*\|+L\|x^k-x^{k-1}\|^2$. Note that $-2L\|x^k-x^{k-1}\|\|x^{k}-x^*\|+L\|x^k-x^{k-1}\|^2\geq-L\|x^k-x^*\|^2\geq-\|x^k-x^*\|^2_{\frac{\widetilde{\Sigma}}{2}}$, where the last inequality holds from the assumption $\Sigma_\nu\succ 2L{\rm{I}}$.
\end{proof}
The next lemma shows that the generated sequence is uniformly bounded, which plays an important role in establishing global convergence property.
\begin{lemma}\label{bounded-iterate}
    Suppose that Assumption \ref{ass1} holds, and for each $\nu=1,2,\ldots,N$ the linear operator $\Sigma_\nu$ satisfies $\Sigma_\nu\succ 2L\rm{I}$, where $L$ is the global Lipschitz constant defined in Assumption \ref{ass1} (v). We also assume that the error sequence $\{d_k\}$ defined in (\ref{sys1}) satisfies $\|d_k\|\leq \epsilon_k$, with $\sum_{k=1}^\infty \epsilon_k<\infty$, and the dual step-size $\tau\in(0,2)$. Let $\{(v^k,\lambda^k)\}$ be the sequence generated by Algorithm \ref{ALM} with $v^k:=(z^k,x^k)$, then the sequence $\{(v^k,\lambda^k)\}$ is bounded in the space $(H\times X)\times X\times (H\times X)$.
\end{lemma}
\begin{proof}
Denote $\Phi_k:=\Phi(v^k,v^{k-1},\lambda^k)$, where $k\geq 1$. By Theorem \ref{Th-ineq}(a), together with $\tau\in(0,2)$ and $\Sigma_\nu\succ 2L{\rm I}$ for all $\nu=1,2,\ldots,N$, we have $\Phi_k-\Phi_{k+1}\geq -2\langle d_k,v^{k+1}-v^*\rangle$. Since $\|d_k\|\leq \epsilon_k$, using the Cauchy-Schwarz inequality we have
$\Phi_k-\Phi_{k+1}\geq -2\langle d_k,v^{k+1}-v^*\rangle
\geq -2\epsilon_k\|v^{k+1}-v^*\|$.

Next, we estimate $\|v^{k+1}-v^*\|$. By Theorem \ref{Th-ineq}(b), the definition of $\widetilde{\Sigma}$ (\ref{sys1}), and the strong monotonicity of each $\Sigma_\nu$, there exists a constant $C>0$ such that
$\|x^k-x^*\|^2\leq C\Phi_k$ and $\|\lambda^k-\lambda^*\|^2\leq C\Phi_k$. On the other hand, from the dual update (\ref{up-lam}) and the definition of $\mathcal{B}$ (\ref{def-Bpsi}), we have $z^{k+1}=
\left(
\mathcal{A}x^{k+1},x^{k+1}
\right)^\top
+\frac{1}{\tau\sigma}(\lambda^k-\lambda^{k+1})$. Since $v^*$ is feasible, there holds
$z^{k+1}-z^*=(\mathcal{A}(x^{k+1}-x^*),x^{k+1}-x^*)^\top
+\frac{1}{\tau\sigma}(\lambda^k-\lambda^{k+1})$.
From the boundedness of $\mathcal{A}$ and the triangle inequality, there exists a constant $C_1$ such that $\|z^{k+1}-z^*\|\leq C_1(\|x^{k+1}-x^*\|+\|\lambda^k-\lambda^*\|+\|\lambda^{k+1}-\lambda^*\|)$.
Due to $\|x^k-x^*\|^2\leq C\Phi_k$ and $\|\lambda^k-\lambda^*\|^2\leq C\Phi_k$, we conclude that there exists a constant ${C}^\prime$ such that $\|z^{k+1}-z^*\|\leq C^\prime(\sqrt{\Phi_k}+\sqrt{\Phi_{k+1}})$. Note that $\|v^{k+1}-v^*\|\leq\|x^{k+1}-x^*\|+\|z^{k+1}-z^*\|$, which further means that there exists a constant $\tilde{C}$ such that $\|v^{k+1}-v^*\|\leq \tilde{C}(\sqrt{\Phi_k}+\sqrt{\Phi_{k+1}})$. Combining with $\Phi_k-\Phi_{k+1}
\geq -2\epsilon_k\|v^{k+1}-v^*\|$, we deduce that 
$\Phi_k-\Phi_{k+1}\geq
-2\tilde{C}\epsilon_k(\sqrt{\Phi_k}+\sqrt{\Phi_{k+1}})$, which means $\sqrt{\Phi_k}+2\tilde{C}\epsilon_k\geq\sqrt{\Phi_{k+1}}$. Then, we deduce that $\sqrt{\Phi_{k+1}}\leq 2\tilde{C}\sum_{i=1}^k\epsilon_i+\sqrt{\Phi_1}$. Combining this with the summability of $\{\epsilon_k\}$, it follows that $\Phi_{k+1}=\Phi(v^{k+1},v^k,\lambda^{k+1})$ is bounded. From the conclusion (b) of Theorem \ref{Th-ineq}, the sequence $\{x^k\}$ and $\{\lambda^k\}$ are bounded. Combining (\ref{up-lam}) with the definition of $\mathcal{B}$ (\ref{def-Bpsi}), we have $z^{k+1}=\frac{\lambda^k-\lambda^{k+1}}{\tau\sigma}+\left(          \mathcal{A}x^{k+1},x^{k+1}\right)^\top$. Due to the boundedness of the operator $\mathcal{A}$, the sequence $\{z^k\}$ is also bounded, and thus the proof is complete.
\end{proof}
\begin{theorem}\label{Convergence-Th}
    Suppose that Assumption \ref{ass1} holds, the linear operators $\Sigma_\nu\succ 2L\rm{I}$ for $\nu=1,2,\ldots,N$, where $L$ is the global Lipschitz constant defined in Assumption \ref{ass1} (v), the dual step-size $\tau\in(0,2)$, and the error sequence $\{d_k\}$ defined in (\ref{sys1}) satisfies $\|d_k\|\leq \epsilon_k$ with $\sum_{k=1}^\infty \epsilon_k<\infty$. Then, for the sequence $\{(v^k,\lambda^k)\}$ with $v^k:=(z^k,x^k)$ generated by Algorithm \ref{ALM}, there holds
\begin{enumerate}
    \item[(a).] any weak limit point of the sequence $\{(v^k,\lambda^k)\}$ is a variational KKT point of the GNEP (\ref{sGNEP-Eq});
    \item[(b).] the whole sequence $\{(v^k,\lambda^k)\}$ converges weakly to a variational KKT point of the GNEP (\ref{sGNEP-Eq}).
    \item[(c).] if the pseudo-gradient operator $P_J$ is strongly monotone, then the sequence $\{x^k\}$ converges strongly to the unique normalized Nash equilibrium of the original GNEP (\ref{sGNEP}).
\end{enumerate}
\end{theorem}
\begin{proof}
 (a). It follows from Lemma \ref{bounded-iterate} that there exists a constant $C$ such that $\|v^{k+1}-v^*\|\leq C$ for all $k$. Combining this with the conclusion (a) of Theorem \ref{Th-ineq}, we have $\frac{1}{2}(\Phi(v^k,v^{k-1},\lambda^k)-\Phi(v^{k+1},v^{k},\lambda^{k+1}))+C\epsilon_k
       \geq \frac{1}{2}\|v^k-v^{k+1}\|^2_{\mathcal{T}}+\left(\frac{1}{\tau}-\frac{1}{2}\right)\frac{1}{\tau\sigma}\|\lambda^k-\lambda^{k+1}\|^2+\frac{1}{2}\|x^k-x^{k+1}\|^2_{\Sigma-2L{\rm{I}}}$.  
Summing both sides from $k=1$ to infinity, we obtain
\begin{equation}\label{ineq-approx}
\begin{aligned}
    &\sum_{k=1}^\infty\left(\frac{1}{2}\|v^k-v^{k+1}\|^2_{\mathcal{T}}+\left(\frac{1}{\tau}-\frac{1}{2}\right)\frac{1}{\tau\sigma}\|\lambda^k-\lambda^{k+1}\|^2+\frac{1}{2}\|x^k-x^{k+1}\|^2_{\Sigma-2L{\rm{I}}}\right)\\
    \leq &\frac{1}{2}\Phi(v^1,v^0,\lambda^1)+C\sum_{k=1}^\infty\epsilon_k<\infty,
    \end{aligned}
\end{equation}
where the inequality follows from the nonnegative property of $\{\Phi(v^k,v^{k-1},\lambda^k)\}$ (the conclusion (b) of Theorem \ref{Th-ineq}) and the summability of $\{\epsilon_k\}$. Since the dual step-size $\tau\in(0,2)$ and $\Sigma_\nu\succ 2L{\rm{I}}$, we conclude that $\|v^k-v^{k+1}\|_{\mathcal{T}}\to 0$, $\|x^k-x^{k+1}\|\to 0$ and $\|\lambda^k-\lambda^{k+1}\|\to 0$ as $k\to \infty$. From the dual update (\ref{up-lam}), there holds $z^{k+1}=\frac{\lambda^k-\lambda^{k+1}}{\tau\sigma}+\left(
         \mathcal{A}x^{k+1},x^{k+1}\right)^\top$, which further means $\|z^{k+1}-z^k\|\to 0$. Thus, we obtain $\|v^k-v^{k+1}\|\to 0$ and $\|\lambda^k-\lambda^{k+1}\|\to 0$. It follows from the global Lipschitz continuity of $P_J(\cdot)$ (Assumption \ref{ass1} (v)) that $\|P_J(x^k)-P_J(x^{k+1})\|\leq L\|x^k-x^{k+1}\|$, and thus $\|P_J(x^k)-P_J(x^{k+1})\|$ tends to zero. Therefore, we conclude that $R_k\to 0$ strongly, where $R_k$ is defined in (\ref{residual-def}).

Due to the uniform boundedness of the sequence $\{(v^k,\lambda^k)\}$ shown by Lemma \ref{bounded-iterate}, and the weak compactness of Hilbert spaces, there exists a subsequence denoted as $\{(v^{k_j},\lambda^{k_j})\}$ that converges weakly to some weak limit point $(\bar{v},\bar{\lambda})$. Notice that $R_k\in F(v^{k+1},\lambda^{k+1})$ given in (\ref{KKT-residual}), $R_k\to 0$ and $(v^{k_j},\lambda^{k_j})\rightharpoonup (\bar{v},\bar{\lambda})$. Combining these facts with the strong-weak sequential closedness of the graph of a maximally monotone operator, $\|(v^{k+1},\lambda^{k+1})-(v^{k},\lambda^{k})\|\to 0$, and Lemma \ref{operatorF-prop}, we conclude that the weak limit point $(\bar{v},\bar{\lambda})$ is a variational KKT point of the GNEP (\ref{sGNEP-Eq}).

(b). Combining the conclusion (a) of Theorem \ref{Th-ineq} with assumptions on $\Sigma_\nu\succ 2L{\rm{I}}$ and the dual step-size $\tau\in(0,2)$, for each weak limit point $(\bar{v},\bar{\lambda})$, we have 
    $\Phi(v^{k+1},v^{k},\lambda^{k+1})\leq \Phi(v^{k},v^{k-1},\lambda^{k})+2\langle d_k,v^{k+1}-\bar{v}\rangle\leq \Phi(v^{k},v^{k-1},\lambda^{k})+C\epsilon_k$, 
where the last inequality follows from the Cauchy-Schwarz inequality and the bounded sequence $\{v^k\}$ obtained in Lemma \ref{bounded-iterate}, and $C$ is a constant such that $\|v^k-\bar{v}\|\leq C$. Due to the nonnegative sequence $\{\Phi(v^{k+1},v^k,\lambda^{k+1})\}$ and the summability of $\{\epsilon_k\}$, it follows from Theorem \ref{thw} that $\{\Phi(v^{k+1},v^k,\lambda^{k+1})\}$ converges to a finite limit. Combining the definition of $\Phi(v^{k+1},v^k,\lambda^{k+1})$ (\ref{def-Phifun}) with $\|x^k-x^{k-1}\|\to 0$ and $\|P_J(x^k)-P_J(x^{k-1})\|\to 0$, we deduce that $\|v^k-\bar{v}\|^2_{\mathcal{T}+\widetilde{\Sigma}}+\frac{1}{\tau\sigma}\|\lambda^k-\bar{\lambda}\|^2$ converges as $k$ tends to infinity. Let us define $\mathcal{N}(v,\lambda):=((\mathcal{T}+\widetilde{\Sigma})^\frac{1}{2}v,\frac{\lambda}{\sqrt{\tau\sigma}})^{\top}$. It follows from Theorem \ref{th: weak-converge} that the sequence $\{\mathcal{N}(v^k,\lambda^k)\}$ is weakly convergent, which means that $\{\lambda^k\}$ weakly converges. Because of the strong monotonicity of each $\Sigma_\nu, (\nu=1,2,\ldots,N)$, we conclude that $\{x^k\}$ is weakly convergent. Notice that $z^{k+1}=\frac{\lambda^k-\lambda^{k+1}}{\tau\sigma}+\left(
         \mathcal{A}x^{k+1},x^{k+1}\right)^\top$ and $\|\lambda^k-\lambda^{k+1}\|\to 0$, we deduce that the sequence $\{z^k\}$ is weakly convergent. Hence, we complete the proof of (b).

(c). From the strong monotonicity of the operator $P_J(\cdot)$, the conclusion obtained in Lemma \ref{operatorF-prop} can be strengthened to $\langle f_1-f_2,(z_1,x_1,\lambda_1)-(z_2,x_2,\lambda_2)\rangle\geq \langle P_J(x_1)-P_J(x_2),x_1-x_2 \rangle\geq c\|x_1-x_2\|^2$, for any $f_1\in F(z_1,x_1,\lambda_1), f_2\in F(z_2,x_2,\lambda_2)$, 
where $c$ is a constant that corresponds to the strong monotonicity of $P_J(\cdot)$. In particular, if there exist two normalized Nash equilibria, denoted as $\bar{x}$ and $\widetilde{x}$, then $0\geq c\|\bar{x}-\widetilde{x}\|^2$. Hence, the normalized Nash equilibrium is unique. Let $\bar{x}$ denote the unique normalized Nash equilibrium. Since $R_k\in F(z^{k+1},x^{k+1},\lambda^{k+1})$, we have 
  $c\|x^{k+1}-\bar{x}\|^2\leq\langle R_k,(v^{k+1}-\bar{v},\lambda^{k+1}-\bar{\lambda})\rangle$.
From the proof of (a), $R_k\to 0$ strongly. Combining this with the uniform boundedness of the sequence $\{(v^k,\lambda^k)\}$ obtained by Lemma \ref{bounded-iterate}, we deduce that $\langle R_k,(v^{k+1}-\bar{v},\lambda^{k+1}-\bar{\lambda})\rangle\to 0$. Hence, we complete the proof.
\end{proof}
\subsection{Convergence Rate}
In this section, we estimate the non-ergodic convergence rate of Algorithm \ref{ALM}. To this end, we need to define a computable residual to reflect the distance between the current iteration and the variational KKT point. Let us first introduce the following notation:
\begin{equation}\label{notations-1}
\begin{aligned}
&F_1(z,x,\lambda):=\left(\begin{array}{c}
        \partial\psi(z)  \\
         0\\
          0
    \end{array}\right),\quad F_2(z,x,\lambda):=\left(
    \begin{array}{c}
          \left(\begin{array}{c}
             0\\
             P_J(x)
        \end{array}\right) +\mathcal{B}^*\lambda\\
      -\mathcal{B}(z,x)
       \end{array}\right).
\end{aligned}
\end{equation}
Notice that the KKT conditions are $0\in F(z,x,\lambda)$ and the operator $F$ is defined in (\ref{def-F}) which satisfies $F=F_1+F_2$. Then, let us introduce the computable KKT residual $\|\mathcal{R}_{k}\|$, where $\mathcal{R}_{k}$ is defined as follows:
\begin{equation}\label{KKT_residual_definition}
    \mathcal{R}_{k}=(z^k,x^k,\lambda^k)-\operatorname{prox}_{F_1}((z^k,x^k,\lambda^k)-F_2(z^k,x^k,\lambda^k))=\left(\begin{array}{c}
        z^k-\operatorname{prox}_\psi(z^k+\lambda^k)  \\
    P_J(x^k)+\mathcal{A}^*\lambda_1^{k}+\lambda_2^{k}\\
        -\mathcal{B}(z^{k},x^{k})
    \end{array}\right).
\end{equation}

\begin{theorem}\label{Th-Rate}
Suppose that Assumption \ref{ass1} holds, the dual step-size $\tau\in (0,2)$, and, for each $\nu=1,2,\ldots,N$, the linear operator $\Sigma_\nu\succ 2L\rm{I}$, where $L$ is the Lipschitz constant in Assumption \ref{ass1} (v). Assume that the error sequence $\{d_k\}$ in (\ref{sys1}) satisfies $\|d_k\|\leq \epsilon_k$ with $\sum_{k=1}^\infty \epsilon_k<\infty$. Let $\{(v^{k},\lambda^k)\}$ be the sequence generated by Algorithm \ref{ALM}, with $v^k:=(z^k,x^k)$. Then, as $K$ tends to infinity, the KKT residual sequence $\{\|\mathcal{R}_{k}\|\}$ satisfies
    $\min_{1\leq k\leq K}\|\mathcal{R}_k\|^2=o\left(\frac{1}{K}\right)$.
\end{theorem}
\begin{proof}
Notice that each iterate $(z^{k+1},x^{k+1},\lambda^{k+1})$ satisfies (\ref{KKT-residual}), where $R_k$ is given in (\ref{residual-def}). Combining with $F=F_1+F_2$ defined in (\ref{notations-1}), we can reformulate (\ref{KKT-residual})
by $(z^{k+1},x^{k+1},\lambda^{k+1})=\operatorname{prox}_{F_1}((z^{k+1},x^{k+1},\lambda^{k+1})+R_k-F_2(z^{k+1},x^{k+1},\lambda^{k+1}))$ equivalently. Due to the non-expansive property of $\operatorname{prox}_{F_1}(\cdot)$, we have $\|\mathcal{R}_{k+1}\|\leq\|R_{k}\|$. Next, let us estimate $\|R_k\|$, from the definition of $R_k$ (\ref{residual-def}) there holds
    \begin{equation*}
    \begin{aligned}
        \|R_k\|^2\leq &5\|d_k\|^2+5\|P_J(x^{k+1})-P_J(x^{k})\|^2+5\|P_J(x^{k})-P_J(x^{k-1})\|^2
       \\ 
       +&5(1-\frac{1}{\tau})^2\|\lambda^k-\lambda^{k+1}\|^2_{\mathcal{B}\mathcal{B}^*}+5\|(\mathcal{T}+\widetilde{\Sigma})(v^k-v^{k+1})\|^2+\frac{\|\lambda^k-\lambda^{k+1}\|^2}{\tau^2\sigma^2}\\
    \leq&5\|d_k\|^2+5L^2\|x^{k+1}-x^k\|^2+5L^2\|x^{k}-x^{k-1}\|^2+5\|\mathcal{T}+\widetilde{\Sigma}\|\|v^k-v^{k+1}\|^2_{\mathcal{T}+\widetilde{\Sigma}}\\
    +&\left(5(1-\frac{1}{\tau})^2\|\mathcal{B}\mathcal{B}^*\|+\frac{1}{\tau^2\sigma^2}\right)\|\lambda^k-\lambda^{k+1}\|^2.
    \end{aligned}
    \end{equation*}
   Then, we analyze each term of the right-hand side separately. Since $\|d_k\|\leq\epsilon_k$ and $\sum_{k=0}^\infty\epsilon_k<\infty$, we have $\sum_{k=1}^\infty\|d_k\|^2<\infty$. 
From the proof of Theorem \ref{Convergence-Th}, we obtain the inequality (\ref{ineq-approx}), which implies
 $ \sum_{k=0}^\infty\|v^k-v^{k+1}\|^2_{\mathcal{T}}<\infty, \sum_{k=0}^\infty\|x^k-x^{k+1}\|^2<\infty$, and $\sum_{k=0}^\infty\|\lambda^k-\lambda^{k+1}\|^2<\infty$.
Notice that $\|v^k-v^{k+1}\|_{\mathcal{T}+\widetilde{\Sigma}}^2=\|v^k-v^{k+1}\|_{\mathcal{T}}^2+\|x^k-x^{k+1}\|^2_\Sigma\leq\|v^k-v^{k+1}\|_{\mathcal{T}}^2+\|\Sigma\|\|x^k-x^{k+1}\|^2$. 
Therefore, we conclude that $\sum_{k=0}^\infty\|R_k\|^2<\infty$. The desired conclusion then follows
immediately.
\end{proof}

 \begin{remark}
  We note that, due to the sGS-induced proximal term, the residual sequence $\{\|\mathcal{R}_k\|\}$ (\ref{KKT_residual_definition}) is not necessarily monotonically decreasing. Thus, we state the non-ergodic convergence rate in the best-iterate sense.
\end{remark}
\section{A Practical sGS-APALM for GNEP (\ref{sGNEP-Eq})}\label{Sec: Pra-ver}
Notice that the implementation of sGS-APALM (Algorithm \ref{ALM}) relies on the construction of the quadratic surrogates in (\ref{quad-approx}), which depends on the choice of $\{\Sigma_\nu\}_{\nu=1}^N$. By Theorem \ref{Convergence-Th}, the convergence is guaranteed if $\Sigma_\nu\succ 2L\rm{I}$, where $L$ is the global Lipschitz constant of the pseudo-gradient mapping $P_J$. In practice, it is always difficult to obtain the constant $L$. To address this issue, we present a practical version of sGS-APALM that adaptively selects $\{\Sigma_\nu\}_{\nu=1}^N$. The resulting method is summarized in Algorithm \ref{adaptive-ALM}.

\begin{algorithm}[ht]
		\caption{A practical sGS-based alternating proximal ALM (sGS-APALM) for GNEP (\ref{sGNEP-Eq})}
		\hspace*{0.02in} {\bf Step 1:} Choose the penalty parameter $\sigma>0$, and the dual step-size $\tau\in (0,2)$. Initialize $\lambda^0\in H\times X$, $z^0\in H\times X$, and $x^0\in X$, set $k=0$ and let $x^{-1}=x^0$. Initialize $\Sigma_\nu=\beta_0{\rm{I}}$, with $\nu=1,2,\ldots,N$, and the factor $\gamma>1$ and parameter $\rho\in(0,1)$.\\
		\hspace*{0.02in} {\bf Step 2a:} Backward Sweeping: For $\nu=N,N-1,\ldots,1$, update $\bar{x}_{\nu}^{k+1}$ by inexactly solving 
		\begin{equation*}
		\min_{x_{\nu}} \hat{L}^{\nu}_{\sigma}(x^k_{\leq\nu-1},x_{\nu},\bar{x}^{k+1}_{\geq\nu+1},z^k;x^k,x^{k-1},\lambda^k).
		\end{equation*}
  \hspace*{0.02in} {\bf Step 2b:} Update $z^{k+1}$ by solving 
		\begin{equation*}
		\min_{z} \psi(z)+\langle\lambda^k,\mathcal{B}(z,\bar{x}^{k+1})\rangle+\frac{\sigma}{2}\|\mathcal{B}(z,\bar{x}^{k+1})\|^2.
		\end{equation*}
  \hspace*{0.02in} {\bf Step 2c:} Forward Sweeping: For $\nu=1,2,\ldots,N$, update $x_{\nu}^{k+1}$ by inexactly solving 
		\begin{equation*}
		\min_{x_{\nu}} \hat{L}^{\nu}_{\sigma} (x^{k+1}_{\leq\nu-1},x_{\nu},\bar{x}^{k+1}_{\geq\nu+1},z^{k+1};x^k,x^{k-1},\lambda^k).
		\end{equation*}
         \hspace*{0.02in} {\bf Step 3:} Compute the pseudo-gradient at $x^{k+1}$ by the notation (\ref{def-gra}). Then, adjust the proximal operators $\Sigma_\nu$, with $\nu=1,2,\ldots,N$ as follows:

    \qquad Case I. If there holds $\|P_J(x^{k+1})-P_J(x^k)\|\leq\frac{\rho}{2}\|x^{k+1}-x^k\|_{\Sigma^2}=\frac{\rho\beta_k}{2}\|x^{k+1}-x^k\|$, accept the iterate $(z^{k+1},x^{k+1})$, go to {\bf Step 4} and keep the proximal operators $\{\Sigma_\nu\}$ by setting $\beta_{k+1}=\beta_k$.
    
    \qquad Case II. If $\|P_J(x^{k+1})-P_J(x^k)\|>\frac{\rho}{2}\|x^{k+1}-x^k\|_{\Sigma^2}=\frac{\rho\beta_k}{2}\|x^{k+1}-x^k\|$, reject the obtained iterate $(z^{k+1},x^{k+1})$, set $\beta_{k}\leftarrow\gamma\beta_k$, and update the proximal operators by $\Sigma_\nu:=\beta_{k}{\rm{I}}$. Return {\bf Step 2a} and recompute the primal iterate $(z^{k+1},x^{k+1})$.

		\hspace*{0.02in} {\bf Step 4:} Update multiplier $\lambda^{k+1}$ by
		\begin{equation*}
	\lambda^{k+1}=\lambda^k+\tau\sigma\mathcal{B}(z^{k+1},x^{k+1}).
		\end{equation*}
		\hspace*{0.02in} {\bf Step 5:}
		 Set $k=k+1$ and return to \textbf{Step 2}.\\
   \label{adaptive-ALM}
	\end{algorithm}
\begin{lemma}\label{beta: stay constant}
    Suppose that Assumption \ref{ass1} holds. Then the parameter sequence $\{\beta_k\}$ defined in Algorithm \ref{adaptive-ALM} is non-decreasing. Moreover, Case II of Step 3 can occur only finitely many times; hence for sufficiently large $k$, $\beta_k$ remains constant. 
\end{lemma}
\begin{proof}
    It follows from Step 3 of Algorithm \ref{adaptive-ALM} that $\{\beta_k\}$ is non-decreasing. Next, we prove the remaining claim by contradiction. Suppose that Case II in Step 3 occurs infinitely many times, then we have $\beta_k\to\infty$. In particular, as $\beta_k$ tends to infinity, there still holds $\|P_J(x^{k+1})-P_J(x^k)\|>\frac{\rho\beta_k}{2}\|x^{k+1}-x^k\|$. However, $P_J$ is assumed to be globally Lipschitz continuous by Assumption \ref{ass1} (v), which contradicts $\beta_k\to\infty$. Hence, Case II in Step 3 can only happen finitely many times, and thus $\beta_k$ remains constant for sufficiently large $k$.
\end{proof}

\begin{theorem}
    Suppose that Assumption \ref{ass1} holds, the dual step-size $\tau\in(0,2)$, and the error sequence $\{d_k\}$ defined in (\ref{sys1}) satisfies $\|d_k\|\leq \epsilon_k$ with $\sum_{k=1}^\infty \epsilon_k<\infty$. Then, for $\{(v^k,\lambda^k)\}$, with $v^k:=(z^k,x^k)$, generated by Algorithm \ref{adaptive-ALM}, we have 
\begin{enumerate}
    \item[(a).] the sequence $\{(v^k,\lambda^k)\}$ converges weakly to a variational KKT point of the GNEP (\ref{sGNEP-Eq}).
    \item[(b).] if the pseudo-gradient operator $P_J$ is strongly monotone, then the sequence $\{x^k\}$ converges strongly to the unique normalized Nash equilibrium.
\end{enumerate}
\end{theorem}
\begin{proof}
    From Lemma \ref{beta: stay constant}, we know that for sufficiently large $k$, there exists $\bar{\beta}$ such that $\beta_k=\bar{\beta}$, and $\|P_J(x^{k+1})-P_J(x^k)\|\leq\frac{\rho\bar{\beta}}{2}\|x^{k+1}-x^k\|$. Hence, the proof proceeds in the same way as Theorem \ref{Convergence-Th} (with $L$ replaced by $\rho\bar{\beta}/2$).  
\end{proof}

\begin{remark}
    Compared with the KKT-residual-based adaptive parameter choice strategies in \cite{EC-21,WWW-25}, our method relies on the pseudo-gradient difference along the primal iterates. This typically yields a much smaller proximal operator sequence and potentially improves the practical convergence rate, which is also supported by our numerical experiments.
\end{remark}

\section{Application to Risk-neutral PDE-constrained GNEPs}\label{Sec: App}
In this section, we apply Algorithm \ref{adaptive-ALM} to a class of risk-neutral PDE-constrained generalized Nash equilibrium problems, which have many applications and have been studied in \cite{GHS-23}. This application serves to illustrate how the proposed abstract framework can be implemented for a specific problem posed in infinite-dimensional spaces.
\subsection{Problem Formulation}
Consider a game involving $N$ players. The optimization problem faced by the $i$-th player is given by:
\begin{equation}\label{probP}\tag{PE}
    \begin{aligned}
        \min_{u_i}\quad &\frac{1}{2}\mathbb{E}_\mathbb{P}[\|y-y_d^i\|^2_{L^2(D)}]+\frac{\nu_i}{2}\|u_i\|^2_{L^2(D)}\\
        \mbox{s.t.}\quad &{\bf{A}}(\omega)y=\sum_{j=1}^N{\bf{B}}_j(\omega)u_j+f(\omega,x),\quad\mathbb{P}-a.s.,\\
        &y\geq\psi,\quad u_i\in U_{ad}^i:=\{u\in L^2(D)\mid a_i\leq u\leq b_i, a.e. \text{ in } D\}, 
    \end{aligned}
\end{equation}
where $D$ is an open bounded set, the random linear operators ${\bf{A}}: \Omega\to\mathcal{L}(H_0^1(D)\cap H^2(D),L^2(D))$ and ${\bf{B}}_j: \Omega\to\mathcal{L}(L^2(D),L^2(D))$ are defined as 
\begin{equation*}
\begin{aligned}
    &\langle {\bf{A}}(\omega)y,v\rangle:=\int_{D}A(x,\omega)\nabla y\cdot\nabla vdx,\quad \text{for all } v\in H_0^1(D)\cap H^2(D),\\
    &\langle {\bf{B}}_j(\omega)u_j,v\rangle:=\int_D [B_j(\omega)u_j]\cdot vdx,\quad \text{for all } v\in L^2(D).
\end{aligned}
\end{equation*}
In addition, the coefficient maps $A: D\times\Omega\to\mathbb{R}$ and $B_j: \Omega\to\mathcal{L}(L^2(D),L^2(D))$ satisfy the regularity conditions given in Assumption 1 of \cite{GHS-23}. Moreover, we also assume the term $f\in L_\mathbb{P}^\infty(\Omega;L^2(D))$, the functions $y_d^i\in L^2(D)$ with $i=1,2,\ldots,N$, and $\psi\in C(\overline{\Omega\times D})$ are given functions. The existence and regularity properties of the normalized KKT points of this problem have been established in Theorem 5 and Theorem 6 of \cite{GHS-23}.
\subsection{Implementation of Algorithm \ref{adaptive-ALM}}\label{Imple}
In this subsection, we detail the implementation of Algorithm \ref{adaptive-ALM} for Problem (\ref{probP}). Without loss of generality, the term $f$ is assumed to be zero; the case $f\neq 0$ can be addressed by the superposition principle. Define the control-to-state operator $S$, which is induced by the linear elliptic random PDE constraints of Problem (\ref{probP}). By Proposition 1 in \cite{GHS-23}, the operator $S$ is well-defined and the state space can be restricted to the Bochner space $L_\mathbb{P}^\infty(\Omega;H_0^1(D)\cap H^2(D))$. Since $H_0^1(D)$ is densely and continuously embedded into $L^2(D)$, the pointwise inequality constraint $y\ge \psi$ can be equivalently formulated in the Hilbert space $L_{\mathbb{P}}^{2}(\Omega;L^{2}(D))$.
 Then, Problem (\ref{probP}) can be rewritten as  
\begin{equation*}
    \begin{aligned}
        \min_{u_i}\quad &\frac{1}{2}\mathbb{E}_\mathbb{P}[\|S(u_i,u_{-i})-y^i_d\|^2_{L^2(D)}]+\frac{\nu_i}{2}\|u_i\|^2_{L^2(D)},\\
        \mbox{s.t.}\quad&S(u_i,u_{-i})\geq\psi\quad\mathbb{P}-a.s.,\quad a_i\leq u_i\leq b_i\quad a.e. \text{ in } D,
    \end{aligned}
\end{equation*}
for $i=1,2,\ldots,N$.
To formulate this problem in the form of Problem (\ref{sGNEP-Eq}), we define
\begin{equation*}
\begin{aligned}
    &J_i(u_i,u_{-i}):=\frac{1}{2}\mathbb{E}_{\mathbb{P}}[\|S(u_i,u_{-i})-y_d^i\|^2_{L^2(D)}]+\frac{\nu_i}{2}\|u_i\|^2_{L^2(D)},\quad \phi_i(u_i):=I_{U_{ad}^i}(u_i),\\
    &\mathcal{K}:=\{y\in H\mid y\geq\psi\quad \mathbb{P}-a.s.\},\quad u:=(u_1,u_2,\dots,u_N),
\end{aligned}
\end{equation*}
where $I_{U_{ad}^i}$ is the indicator function of $U_{ad}^i$,
with the spaces given by \\
$H:=L_\mathbb{P}^2(\Omega;L^2(D))$, $Y:=L_\mathbb{P}^\infty(\Omega;H_0^1(D)\cap H^2(D))$, and $X=(L^2(D))^N$.

Moreover, for each $i=1,2,\ldots,N$, we define the linear operator $S_i: L^2(D)\to L^2(\Omega;L^2(D))$ as the solution operator of the random PDE ${\bf{A}}(\omega)y={\bf{B}}_iu_i$. By the linearity of the PDE constraint, there holds $S(u_1,u_2,\ldots,u_N)=\sum_{i=1}^NS_i(u_i)$.

Then, the original problem (\ref{probP}) can be rewritten as follows:\\
      $  \min_{(u_i,z_1,z_{2,i})} J_i(u_i,u_{-i})+I_{\mathcal{K}}(z_1)+\phi_i(z_{2,i})\quad
        \mbox{ s.t. } \sum_{i=1}^NS_iu_i-z_1=0, u_i-z_{2,i}=0$.\\
To be consistent with the notation of Problem (\ref{sGNEP-Eq}), we define $\mathcal{A}u:=\sum_{i=1}^NS_iu_i$ and define the operator $\mathcal{B}$ by (\ref{def-Bpsi}) accordingly.

To apply Algorithm \ref{adaptive-ALM}, we first verify the satisfaction of Assumption \ref{ass1}. Note that items (i), (ii), (iii), (vi) and the global Lipschitz continuity of the pseudo-gradient of Assumption \ref{ass1} are satisfied by the structure of the problem and the regularity properties of $S$ discussed in Corollary 1 of \cite{GHS-23}. Due to the fact that $\langle P_J(u)-P_J(\tilde{u}),u-\tilde{u} \rangle_X=\mathbb{E}_\mathbb{P}[\|u-\tilde{u}\|^2_{S^*S+\operatorname{diag}(\nu_1\rm{I},\nu_2\rm{I},\ldots,\nu_N\rm{I})}]\geq 0$, we validate Assumption \ref{ass1} (v). In addition, Assumption \ref{ass1} (iv) is standard in PDE optimal control context; see Section 2.7 of \cite{MR-09} for more discussions. 

Notice that Algorithm \ref{adaptive-ALM} is posed in infinite-dimensional spaces, and must be discretized for computation. Our numerical implementation follows the ``first-optimize-then-discretize" paradigm, which consists of first designing the algorithm in infinite-dimensional spaces, and then discretizing the scheme to obtain a practical numerical method. The discretization of the corresponding spaces has two components. For the spatial discretization, we employ the piecewise linear finite element method on a uniform mesh. For the random field, we apply a standard Monte Carlo approximation. Specifically, given i.i.d. samples $\{\omega_m\}_{m=1}^M$, the expectation operator $\mathbb{E}_{\mathbb{P}}[\cdot]$ is approximated by the empirical average, i.e., $\mathbb{E}_{\mathbb{P}}[g(\omega)]\approx\frac{\sum_{m=1}^Mg(\omega_m)}{M}$. 
\subsection{Numerical Experiments}
In this section, we present numerical results to demonstrate the efficiency of the proposed Algorithm \ref{adaptive-ALM}. All algorithms were implemented in MATLAB R2023a and executed on a laptop equipped with an Intel(R) Core(TM) i7-12650H CPU (2.30GHz) and 16GB of RAM.

It is important to note that the numerical literature on risk-neutral GNEP constrained by PDEs is still limited. To the best of our knowledge, the most closely related method is the Moreau-Yosida regularization method proposed in in \cite{GHS-23} for Problem (\ref{probP}). However, \cite{GHS-23} focuses mainly on theoretical analysis, and its implementation faces substantial challenges. As discussed in Section 5.1 of \cite{GHS-23}, at each iteration their method requires solving a coupled risk-neutral Nash equilibrium problem (NEP) with high accuracy. This subproblem is computationally expensive and often ill-conditioned due to the coupling of all players and the presence of PDE constraints. Moreover, the numerical results reported in Fig. 1 of \cite{GHS-23} suggest that the method may be computationally demanding. In view of these computational difficulties and the fundamentally different algorithmic structure, we do not compare our method with \cite{GHS-23}.

Recall the KKT residual defined in (\ref{KKT_residual_definition}). The algorithms are terminated when the residual denoted as $\eta_{KKT}$ falls below a given tolerance. Specifically, we require:
\begin{center}
$\eta_{KKT}:=\max\{\eta_p,\eta_d,\eta_{I_1},\eta_{I_2}\}\leq tol=10^{-5}$,
\end{center}
where the residual components are given as:
\begin{equation*}
    \begin{aligned}
    &\eta_p:=\max\{\sqrt{\mathbb{E}_\mathbb{P}(\|S(u^k)-z_1^k\|^2)},\|u^k-z_2^k\|\}, \eta_{d}:=\|P_J(u^k)+S^*\lambda^k_1+\lambda_2^k\|,\\
    &\eta_{I_1}:=\sqrt{\mathbb{E}_\mathbb{P}(\|z_1^k-\max(z_1^k+\lambda_1^k,\psi)\|^2)}, \eta_{I_2}:=\|z_2^k-\operatorname{Pr}_{U_{ad}}(z_2^k+\lambda_2^k)\|,
    \end{aligned}
\end{equation*}
with $\operatorname{Pr}_{U_{ad}}(\cdot)$ denoting the projection operator onto the set $U_{ad} := \Pi_{i=1}^N U_{ad}^i$. 
The maximum number of iterations is set to $20,000$.

For all examples, the initial guess is set to $(u^0,z^0,\lambda^0)=0$. The initial proximal operators in Algorithm \ref{adaptive-ALM} are chosen as $\Sigma_\nu := \mathrm{I}$ for $\nu=1,\dots,N$, with an increasing factor $\gamma = 1.2$ and the parameter $\rho=0.99$. In our implementation, the subproblems in the backward Gauss-Seidel sweeping are solved exactly, while we decide whether to skip the forward Gauss-Seidel sweeping by checking the inexactness condition with error tolerance $\epsilon_k \leq k^{-1.8}$ (as described in Theorem \ref{Convergence-Th}). To validate the efficiency of the proposed algorithm, we compare the performance of the following algorithms:
\begin{enumerate}
    \item[(i)] \textbf{isGS-APALM-$\tau$:} The proposed Algorithm \ref{adaptive-ALM}, implemented with dual step-sizes $\tau \in \{1, 1.618, 1.9\}$.
    \item[(ii)] \textbf{sGS-APALM-$\tau$:} The exact variant of Algorithm \ref{adaptive-ALM} where all the subproblems are solved exactly. We include this benchmark (using the same step-sizes $\tau \in \{1, 1.618, 1.9\}$) specifically to demonstrate the computational advantage of the proposed inexact strategy.
    \item[(iii)] \textbf{multi-ADMM:} The extension of multiple-block ADMM-type method proposed in \cite{EC-21}. The main motivation for this comparison is that this method is also designed for GNEPs in general settings and shares some similar structural features with Algorithm \ref{adaptive-ALM}. Moreover, numerical results in \cite{EC-21} have shown that classical forward–backward schemes often converge more slowly in practice than this type of method.
\end{enumerate}

For clarity, we introduce some notation to represent the objective values and the violation of the probability constraints:
\begin{center}
$J_i:=\frac{1}{2}\mathbb{E}_\mathbb{P}[\|y-y^i_d\|^2_{L^2(D)}]+\frac{\nu_i}{2}\|u_i\|^2_{L^2(D)}$, $i=1,2,\ldots,N$,\\
$\mathbb{P}_{vio}:=\mathbb{P}(\{\omega\in\Omega\mid\int_D \min(y(x,\omega)-\psi,0)dx\neq0\})$.
\end{center}
We also report the KKT residual denoted as $\eta_{KKT}$ and the final proximal operator $\beta{\rm{I}}$ in our numerical results.

\textbf{Example 1.} We first consider the example given in Section 5.3.2 of \cite{GHS-23}, which involves $N=2$ players.
\begin{equation*}
    \begin{aligned}
    \min_{u_i}\quad &\frac{1}{2}\mathbb{E}_\mathbb{P}[\|y-y^i_d\|^2_{L^2(D)}]+\frac{\nu_i}{2}\|u_i\|^2_{L^2(D)}\\
        \mbox{s.t.}\quad &-\nu(\omega)\partial_{xx}y(\omega,x)=\sum_{i=1}^2u_i+f(\omega,x),\text{ in } D, \text{ for }a.e. \omega\in\Omega,\\
        &y(\omega,0)=d_0(\omega), y(\omega,1)=d_1(\omega),\\
        &y(\omega,x)\geq 0 \quad\mathbb{P}-a.s.,\quad u_i\in U_{ad}^i:=\{u\in L^2(D)\mid a_i\leq u\leq b_i, a.e. \text{ in } D\},
    \end{aligned}
\end{equation*}
where the domain $D=(0,1)$, the regularization parameters $\nu_1=\nu_2=10^{-3}$, and the parameters in the random PDE are defined by
        $\nu(\omega):=\max(0.05,\xi_1(\omega)),\quad f(\omega,x):=\frac{2\xi_2(\omega)-1}{10}$,
        $d_0(\omega):=0.75+\frac{2\xi_3(\omega)-1}{1000},\quad d_1(\omega):=0.5+\frac{\xi_4(\omega)}{1000}$,
where the random variables $\xi_i: \Omega\to\mathbb{R}, i=1,2,3,4$ are i.i.d. and uniformly distributed on $[0,1]$. The target state functions are set by $y_d^1:=\sin(\frac{50x}{\pi})$ and $y_d^2:=\cos(\frac{50x}{\pi})$, and the control bounds are taken to be $a_1=a_2=-1$ and $b_1=b_2=1$.

We set the Monte Carlo sample size to $M=2000$ and test the numerical behavior of Algorithm \ref{adaptive-ALM} on different mesh sizes $h=2^{-i}, i=5,6,7,8$. As discussed in the previous subsection \ref{Imple}, updating the primal variables requires solving a series of linear equations, which are obtained by discretizing the linear operators $\sigma_1\mathbb{E}_{\mathbb{P}}(S_\nu^*S_\nu)+\sigma_2{\rm{I}}+\Sigma_\nu$ for $\nu=1,2,\ldots,N$. Meanwhile, updating the dual variables requires computing the corresponding states, which amounts to solving the elliptic PDEs for each sampled parameter $\{\omega_m\}_{m=1}^M$. This entails solving $M$ linear equations. In our implementations, these linear equations are directly solved by the default backslash `\textbackslash' in MATLAB. The penalty parameters are chosen as $\sigma_1=10^2$ for the constraint $y=z_1$ and $\sigma_2=1$ for the constraint $u=z_2$. For the implementation of ``multi-ADMM", we follow all the details discussed in Section 6.2 of \cite{EC-21}, and select the same penalty parameters as those of our designed method. 

We report detailed numerical results in Table \ref{tab1}. It is observed that the iteration numbers of the proposed method do not increase greatly as the mesh size becomes finer. We also note that ``multi-ADMM" fails to reach the prescribed KKT tolerance within the iteration limit, and the final proximal term is much larger than that of our method. The main reasons are that (i) the convergence of ``multi-ADMM" requires a sufficiently large proximal term, as shown in Theorem 4.6 and Section 4.3 of \cite{EC-21}; (ii) the adaptive parameter choice strategy designed in \cite{EC-21} is based on KKT residuals, more specifically, the proximal operators are enlarged when the KKT residual fails to decrease. However, the KKT residual may not be decreasing even when the iteration number is large. By comparison, the convergence of our proposed Algorithm \ref{adaptive-ALM} requires a theoretically smaller proximal term. More specifically, the adaptive strategy uses only local differences of the pseudo-gradient.

The CPU times further show that (i) solving each subproblem inexactly, and (ii) using a larger step-size both help us improve the practical efficiency. An interesting observation is that in almost all cases the inexact variants require much fewer iterations than their exact counterparts. Similar results have also been observed when using inexact sGS-based ADMM to handle large-scale SDPs, which is reported in Table 3 of \cite{LXD-21} in detail. In our numerical implementations, we note that in the early phase, each subproblem can be solved to a loose tolerance, and thus some computations of the forward sweeping procedure can be skipped. As the iterations proceed, the required accuracy of each subproblem solver becomes more stringent, which makes the skipped computations no longer admissible, and the algorithm gradually transitions to the exact scheme. In this sense, the early inexact iterations can be viewed as a warm start for the subsequent exact phase, which helps explain why the iteration number of the inexact algorithm is typically smaller in our numerical results.

The other contributing factor is that our theoretical results allow the dual step-size $\tau$ to be chosen in $(0,2)$. Comparing the CPU times of different step-size choices, we notice that both ``isGS-APALM-1.9" and ``isGS-APALM-1.618" improve the efficiency of ``isGS-APALM-1", except for the case $h=\frac{1}{2^7}$. This aligns with our theoretical results and shows that choosing a larger dual step-size generally accelerates the practical convergence rate. Similar numerical behavior is also reported in Section 5.1 of \cite{LXD-21}, but for completely different problems.

The values of the objective functions $J_1$ and $J_2$, together with the state-constraint feasibility indicator $\mathbb{P}_{vio}$, imply that the obtained solutions are reliable. In particular, the value of $\mathbb{P}_{vio}$ shows that the probability of violation of the state constraints is acceptable, since the usual tolerance for feasibility violation is set by $5\%$ (discussed in Section 5 of \cite{GHS-23}). The numerical results are shown in Figure \ref{fig1}.
\begin{table}[htbp]
	 	\caption{Numerical results of Algorithm \ref{adaptive-ALM} for Example 1}
	 	\centering
   {\tiny
	 	\begin{tabular}{|c|c |c| c| c| c| c|c |}
	 		\hline
	 		$h$&Algorithm &Iter & CPU(sec) &$J_1$&$J_2$&$\mathbb{P}_{vio}$&$\eta_{KKT}\mid\beta$\\
            \hline  
            \multirow{7}{*}{$\frac{1}{2^5}$}&isGS-APALM-1.9&5925&17.57&$3.27\times 10^{-1}$& $4.01\times 10^{-1}$&$0.05\%$&$9.9\times 10^{-6}\mid 1.2$ \\
            \multirow{7}{*}{~}&isGS-APALM-1.618&5823&17.01&$3.27\times 10^{-1}$& $4.01\times 10^{-1}$&$0.05\%$&$9.9\times 10^{-6}\mid 1.2$ \\
             \multirow{7}{*}{~}&isGS-APALM-1&7302& 22.92&$3.27\times 10^{-1}$& $4.01\times 10^{-1}$&$0.05\%$&$9.9\times 10^{-6}\mid 1.0$ \\
             \multirow{7}{*}{~}&sGS-APALM-1.9&7290&22.11&$3.27\times 10^{-1}$& $4.01\times 10^{-1}$&$0.05\%$&$9.9\times 10^{-6}\mid 1.0$ \\
            \multirow{7}{*}{~}&sGS-APALM-1.618&7590&23.10&$3.27\times 10^{-1}$& $4.01\times 10^{-1}$&$0.05\%$& $9.9\times 10^{-6}\mid 1.2$\\
             \multirow{7}{*}{~}&sGS-APALM-1&7825&24.36&$3.27\times 10^{-1}$& $4.01\times 10^{-1}$&$0.05\%$& $9.9\times 10^{-6}\mid 1.0$\\
             \multirow{7}{*}{~}&multi-ADMM&$-$&$-$&$-$&$-$&$-$& $9.6\times 10^{-3}\mid 17101.1$ \\
            \hline 
            \multirow{7}{*}{$\frac{1}{2^6}$}&isGS-APALM-1.9&6962&22.75&$3.31\times 10^{-1}$& $4.10\times 10^{-1}$&$0.05\%$&$9.8\times 10^{-6}\mid$1.2 \\
            \multirow{7}{*}{~}&isGS-APALM-1.618&7635& 26.66&$3.31\times 10^{-1}$& $4.10\times 10^{-1}$&$0.05\%$&$9.8\times 10^{-6}\mid$1.2 \\
             \multirow{7}{*}{~}&isGS-APALM-1&9144& 30.98&$3.31\times 10^{-1}$& $4.10\times 10^{-1}$&$0.05\%$& $9.9\times 10^{-6}\mid$1.0\\
             \multirow{7}{*}{~}&sGS-APALM-1.9&8984&30.77&$3.31\times 10^{-1}$& $4.10\times 10^{-1}$&$0.05\%$& $9.9\times 10^{-6}\mid$1.2\\
             \multirow{7}{*}{~}&sGS-APALM-1.618&8532&28.88&$3.31\times 10^{-1}$& $4.10\times 10^{-1}$&$0.05\%$&$9.8\times 10^{-6}\mid$1.2 \\
             \multirow{7}{*}{~}&sGS-APALM-1&8181&28.01&$3.31\times 10^{-1}$& $4.10\times 10^{-1}$&$0.05\%$&$9.9\times 10^{-6}\mid$1.0 \\
             \multirow{7}{*}{~}&multi-ADMM&$-$&$-$&$-$&$-$&$-$&$1.0\times 10^{-2}\mid$17212.1 \\
            \hline 
            \multirow{7}{*}{$\frac{1}{2^7}$}&isGS-APALM-1.9&7441&109.45&$3.32\times 10^{-1}$& $4.15\times 10^{-1}$&$0.05\%$&$9.8\times 10^{-6}\mid$1.0 \\
             \multirow{7}{*}{~}&isGS-APALM-1.618& 6364 &95.38&$3.32\times 10^{-1}$& $4.15\times 10^{-1}$&$0.05\%$&$9.8\times 10^{-6}\mid$1.44 \\
             \multirow{7}{*}{~}&isGS-APALM-1&6524& 98.82&$3.32\times 10^{-1}$& $4.15\times 10^{-1}$&$0.05\%$& $9.8\times 10^{-6}\mid$1.0 \\
             \multirow{7}{*}{~}&sGS-APALM-1.9&9067&137.36&$3.32\times 10^{-1}$& $4.15\times 10^{-1}$&$0.05\%$&$9.8\times 10^{-6}\mid$1.2  \\
             \multirow{7}{*}{~}&sGS-APALM-1.618&8769 &130.18&$3.32\times 10^{-1}$& $4.15\times 10^{-1}$&$0.05\%$&$9.8\times 10^{-6}\mid$1.0  \\
             \multirow{7}{*}{~}&sGS-APALM-1&9692&141.44&$3.32\times 10^{-1}$& $4.15\times 10^{-1}$&$0.05\%$&$9.8\times 10^{-6}\mid$1.2  \\
             \multirow{7}{*}{~}&multi-ADMM&$-$&$-$&$-$&$-$&$-$&$9.9\times 10^{-3}\mid$17193.1  \\
            \hline
                 \multirow{7}{*}{$\frac{1}{2^8}$}&isGS-APALM-1.9&7054&251.45&$3.31\times 10^{-1}$& $4.14\times 10^{-1}$&$0.05\%$& $9.8\times 10^{-6}\mid$1.0\\
              \multirow{7}{*}{~}&isGS-APALM-1.618& 7158 &256.63&$3.31\times 10^{-1}$& $4.14\times 10^{-1}$&$0.05\%$& $9.8\times 10^{-6}\mid$1.2\\
             \multirow{7}{*}{~}&isGS-APALM-1&7689& 275.00&$3.31\times 10^{-1}$& $4.14\times 10^{-1}$&$0.05\%$& $9.8\times 10^{-6}\mid$1.0\\
             \multirow{7}{*}{~}&sGS-APALM-1.9&9019&327.06&$3.31\times 10^{-1}$& $4.14\times 10^{-1}$&$0.05\%$&$9.8\times 10^{-6}\mid$1.2 \\
             \multirow{7}{*}{~}&sGS-APALM-1.618& 8243 &295.06&$3.31\times 10^{-1}$& $4.14\times 10^{-1}$&$0.05\%$& $9.8\times 10^{-6}\mid$1.0\\
             \multirow{7}{*}{~}&sGS-APALM-1&8511&307.39&$3.31\times 10^{-1}$& $4.14\times 10^{-1}$&$0.05\%$&$9.8\times 10^{-6}\mid$1.0 \\
             \multirow{7}{*}{~}&multi-ADMM&$-$&$-$&$-$&$-$&$-$& $1.0\times 10^{-2}\mid$17196.1\\
             \hline
	 	\end{tabular}}%
	 	\label{tab1}
	 \end{table}

\begin{figure}[htbp]
	 	\centering
	 	\subfigure[Optimal control $u_1$]{
	 		\begin{minipage}[t]{4.0cm}
	 			\centering
	 			\includegraphics[width=\linewidth]{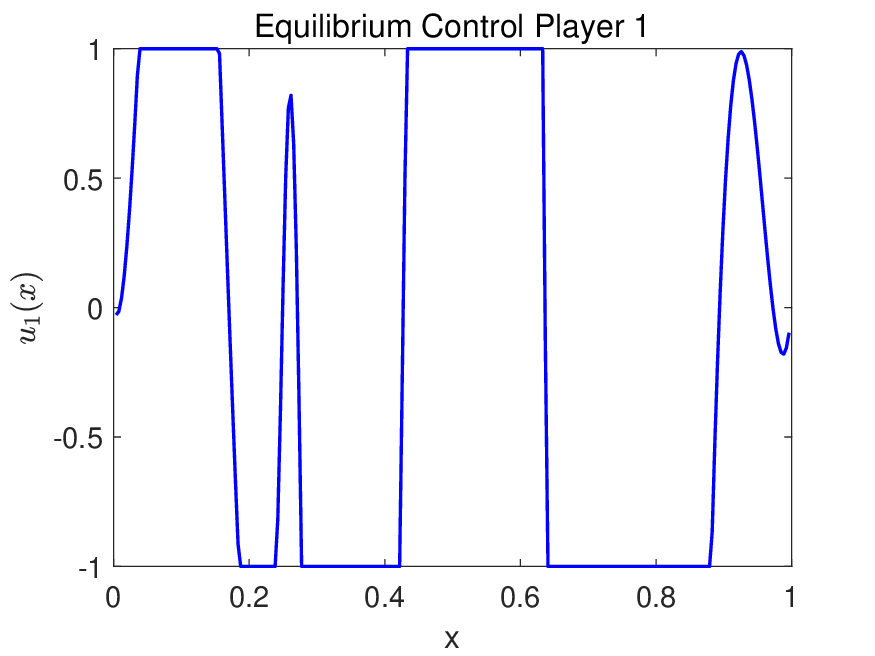}	
	 	\end{minipage}}
	 	\subfigure[Optimal control $u_2$]{
	 		\begin{minipage}[t]{4.0cm}
	 			\centering
	 			\includegraphics[width=\linewidth]{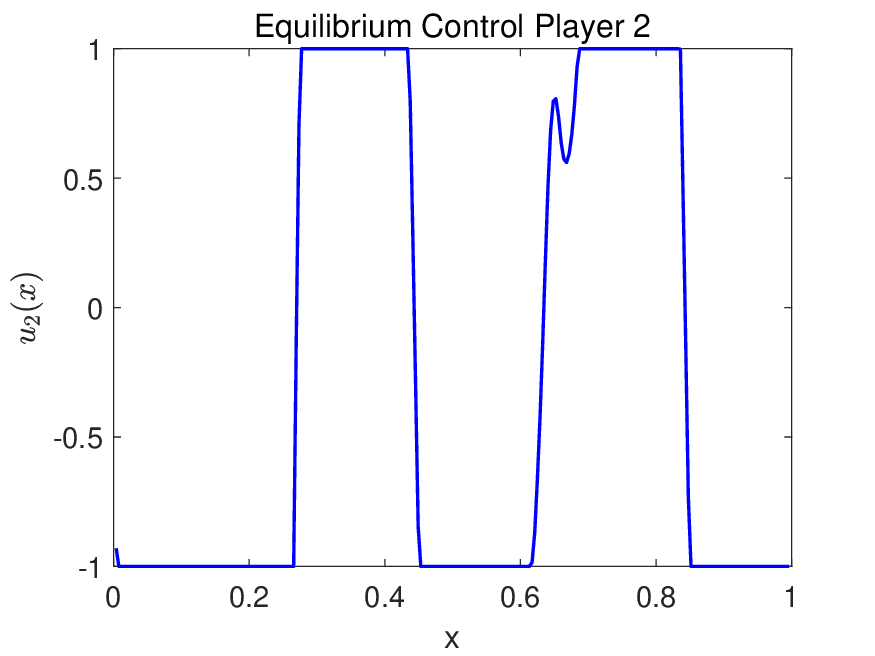}
	 	\end{minipage}}
	 	\subfigure[Sample states $y(\cdot,\omega)$]{
	 		\begin{minipage}[t]{4.0cm}
	 			\centering
	 			\includegraphics[width=\linewidth]{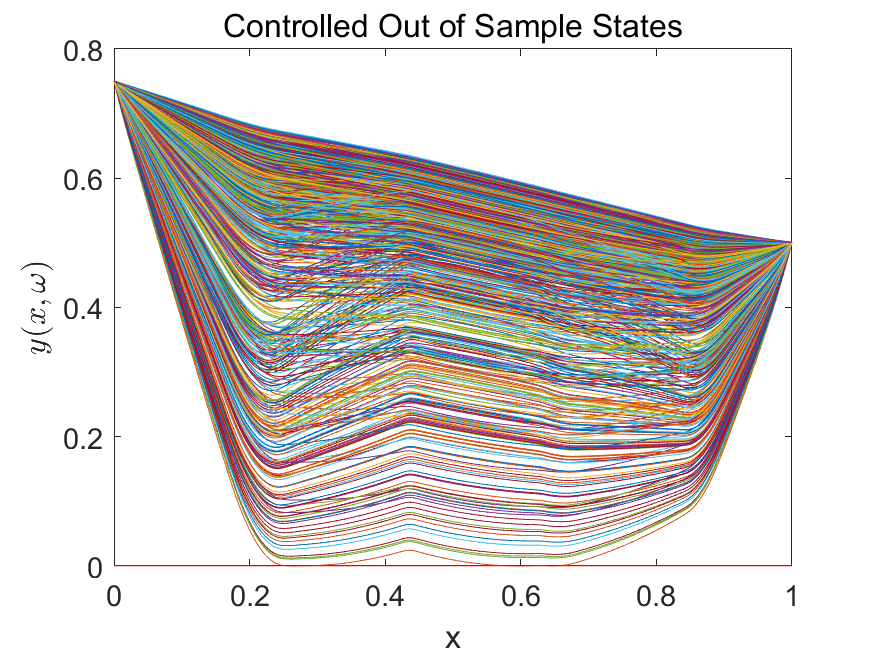}	
	 	\end{minipage}}
	 	\caption{\centering Numerical solution obtained by Algorithm \ref{adaptive-ALM} with $h=\frac{1}{2^8}$ of Example 1}
	 	\label{fig1}
	 \end{figure}

\textbf{Example 2.} Then, we consider a general case with $N=2$ players, where the control region of each player is a subset of $D$. The problem is given by:
\begin{equation*}
    \begin{aligned}
    \min_{u_i}\quad &\frac{1}{2}\mathbb{E}_\mathbb{P}[\|y-y^i_d\|^2_{L^2(D)}]+\frac{\nu_i}{2}\|u_i\|^2_{L^2(D)}\\
        \mbox{s.t.}\quad &-\nu(\omega)\partial_{xx}y(\omega,x)=\sum_{i=1}^2B_iu_i+f(\omega,x),\text{ in } D, \text{for } a.e. \omega\in\Omega,\\
        &y(\omega,0)=d_0(\omega), y(\omega,1)=d_1(\omega), \\
        &y(\omega,x)\geq 0 \quad\mathbb{P}-a.s.,\quad u_i\in U_{ad}^i:=\{u\in L^2(D)\mid a_i\leq u\leq b_i, a.e. \text{ in } D\},
    \end{aligned}
\end{equation*}
where $B_1u_1:=\chi_{[0,\frac{1}{2})}\cdot u_1$ and $B_2u_2:=\chi_{[\frac{1}{2},1)}\cdot u_2$. The control bounds are given by $a_1=a_2=-\frac{3}{4}$ and $b_1=b_2=\frac{3}{4}$. The other parameters of the problem are defined in the same way as in Example 1.

As in Example 1, we set the penalty parameters by $\sigma_1=10^2$ for the constraint $y=z_1$ and $\sigma_2=1$ for the constraint $u=z_2$. The linear systems encountered are again solved by the default backslash operator `\textbackslash'. Detailed numerical results are reported in Table \ref{tab2}. From Table \ref{tab2}, we observe that for all cases the designed ``isGS-APALM-1.9" has the best performance. Moreover, for both inexact algorithms and their exact counterparts, using a larger dual step-size, i.e., $\tau=1.618$ and $\tau=1.9$, can improve numerical efficiency relative to that of $\tau=1$. This again validates that selecting a larger dual step-size can help us obtain a faster practical convergence rate. Moreover, solving each subproblem inexactly makes the algorithm more flexible and reduces the total computational cost. In addition, we note that ``multi-ADMM" fails to reach the prescribed KKT tolerance within the maximum iteration limit. This is consistent with the theoretical advantages of the proposed method, specifically, (i) the convergence of the designed method is guaranteed under much weaker conditions than those of \cite{EC-21}; (ii) the adaptive strategy enables us to add a relatively small proximal term.

Notice that the resulting iteration counts are mildly affected by the refinement of the mesh sizes. Additionally, the probability of violating the state constraints is well within the typical tolerance for probability constraints (which is usually $5\%$). These observations further indicate that the obtained solutions are satisfactory. The computed results are depicted in Figure \ref{fig2}.     

\begin{table}[htbp]
	 	\caption{Numerical results of Algorithm \ref{adaptive-ALM} for Example 2}
	 	\centering
   {\tiny
	 	\begin{tabular}{|c|c |c| c| c| c| c| c|}
	 		\hline
	 		$h$&Algorithm &Iter & CPU(sec) &$J_1$& $J_2$&$\mathbb{P}_{vio}$&$\eta_{KKT}\mid\beta$\\
             \hline 
            \multirow{7}{*}{$\frac{1}{2^5}$}&isGS-APALM-1.9&7249&10.16&$3.29\times 10^{-1}$& $4.01\times 10^{-1}$&$0.05\%$& $9.9\times 10^{-6}\mid$1.0\\
              \multirow{7}{*}{~}&isGS-APALM-1.618&7478 &11.33 &$3.29\times 10^{-1}$& $4.01\times 10^{-1}$&$0.05\%$& $9.9\times 10^{-6}\mid$1.0\\
             \multirow{7}{*}{~}&isGS-APALM-1&7548& 11.85&$3.29\times 10^{-1}$& $4.01\times 10^{-1}$&$0.05\%$& $9.9\times 10^{-6}\mid$1.0\\
             \multirow{7}{*}{~}&sGS-APALM-1.9&7290&16.03&$3.29\times 10^{-1}$& $4.01\times 10^{-1}$&$0.05\%$& $9.9\times 10^{-6}\mid$1.0\\
             \multirow{7}{*}{~}&sGS-APALM-1.618& 7532& 16.87&$3.29\times 10^{-1}$& $4.01\times 10^{-1}$&$0.05\%$& $9.9\times 10^{-6}\mid$1.0\\
             \multirow{7}{*}{~}&sGS-APALM-1&7686&17.10&$3.29\times 10^{-1}$& $4.01\times 10^{-1}$&$0.05\%$& $9.9\times 10^{-6}\mid$1.0\\
             \multirow{7}{*}{~}&multi-ADMM&$-$&$-$&$-$&$-$&$-$& $1.8\times 10^{-3}\mid$1455.1\\
            \hline 
            \multirow{7}{*}{$\frac{1}{2^6}$}&isGS-APALM-1.9&5730&25.46&$3.31\times 10^{-1}$& $4.11\times 10^{-1}$&$0.05\%$& $9.8\times 10^{-6}\mid$1.0\\
             \multirow{7}{*}{~}&isGS-APALM-1.618&6261 & 31.81&$3.31\times 10^{-1}$& $4.11\times 10^{-1}$&$0.05\%$& $9.8\times 10^{-6}\mid$1.0\\
             \multirow{7}{*}{~}&isGS-APALM-1&8319& 42.13&$3.31\times 10^{-1}$& $4.11\times 10^{-1}$&$0.05\%$& $9.8\times 10^{-6}\mid$1.0\\
             \multirow{7}{*}{~}&sGS-APALM-1.9&5780&29.38&$3.31\times 10^{-1}$& $4.11\times 10^{-1}$&$0.05\%$& $9.8\times 10^{-6}\mid$1.0\\
             \multirow{7}{*}{~}&sGS-APALM-1.618& 6422&34.56 &$3.31\times 10^{-1}$& $4.11\times 10^{-1}$&$0.05\%$& $9.8\times 10^{-6}\mid$1.0\\
             \multirow{7}{*}{~}&sGS-APALM-1&8317&40.27&$3.31\times 10^{-1}$& $4.11\times 10^{-1}$&$0.05\%$& $9.8\times 10^{-6}\mid$1.0\\
             \multirow{7}{*}{~}&multi-ADMM&$-$&$-$&$-$&$-$&$-$& $1.5\times 10^{-3}\mid$1154.1\\
            \hline 
            \multirow{7}{*}{$\frac{1}{2^7}$}&isGS-APALM-1.9&6882&103.44&$3.32\times 10^{-1}$& $4.14\times 10^{-1}$&$0.05\%$& $9.8\times 10^{-6}\mid$1.0\\
              \multirow{7}{*}{~}&isGS-APALM-1.618&7134 & 109.08&$3.32\times 10^{-1}$& $4.14\times 10^{-1}$&$0.05\%$& $9.8\times 10^{-6}\mid$1.0\\
             \multirow{7}{*}{~}&isGS-APALM-1&7879& 121.28&$3.32\times 10^{-1}$& $4.14\times 10^{-1}$&$0.05\%$& $9.8\times 10^{-6}\mid$1.0\\
             \multirow{7}{*}{~}&sGS-APALM-1.9&6980&108.77&$3.32\times 10^{-1}$& $4.14\times 10^{-1}$&$0.05\%$& $9.8\times 10^{-6}\mid$1.0\\
              \multirow{7}{*}{~}&sGS-APALM-1.618&6911 & 106.33&$3.32\times 10^{-1}$& $4.14\times 10^{-1}$&$0.05\%$& $9.8\times 10^{-6}\mid$1.0\\
             \multirow{7}{*}{~}&sGS-APALM-1&7946&131.14&$3.32\times 10^{-1}$& $4.14\times 10^{-1}$&$0.05\%$& $9.8\times 10^{-6}\mid$1.0\\
             \multirow{7}{*}{~}&multi-ADMM&$-$&$-$&$-$&$-$&$-$& $1.7\times 10^{-3}\mid$1157.1\\
            \hline
             \multirow{7}{*}{$\frac{1}{2^8}$}&isGS-APALM-1.9&6216&170.81&$3.34\times 10^{-1}$& $4.18\times 10^{-1}$&$0.05\%$& $9.8\times 10^{-6}\mid$1.0\\
             \multirow{7}{*}{~}&isGS-APALM-1.618& 6981 & 190.87&$3.34\times 10^{-1}$& $4.18\times 10^{-1}$&$0.05\%$& $9.8\times 10^{-6}\mid$1.0\\
             \multirow{7}{*}{~}&isGS-APALM-1&7570& 211.59&$3.34\times 10^{-1}$& $4.18\times 10^{-1}$&$0.05\%$& $9.8\times 10^{-6}\mid$1.0\\
             \multirow{7}{*}{~}&sGS-APALM-1.9&6239&181.39&$3.34\times 10^{-1}$& $4.18\times 10^{-1}$&$0.05\%$& $9.8\times 10^{-6}\mid$1.0\\
             \multirow{7}{*}{~}&sGS-APALM-1.618& 7028& 200.11&$3.34\times 10^{-1}$& $4.18\times 10^{-1}$&$0.05\%$& $9.8\times 10^{-6}\mid$1.0\\
             \multirow{7}{*}{~}&sGS-APALM-1&7509&210.23&$3.34\times 10^{-1}$& $4.18\times 10^{-1}$&$0.05\%$& $9.8\times 10^{-6}\mid$1.0\\
             \multirow{7}{*}{~}&multi-ADMM&$-$&$-$&$-$&$-$&$-$& $1.7\times 10^{-3}\mid$1149.1\\
             \hline
	 	\end{tabular}}%
	 	\label{tab2}
	 \end{table}

 \begin{figure}[htbp]
	 	\centering
	 	\subfigure[Optimal control $u_1$]{
	 		\begin{minipage}[t]{4.0cm}
	 			\centering
	 			\includegraphics[width=\linewidth]{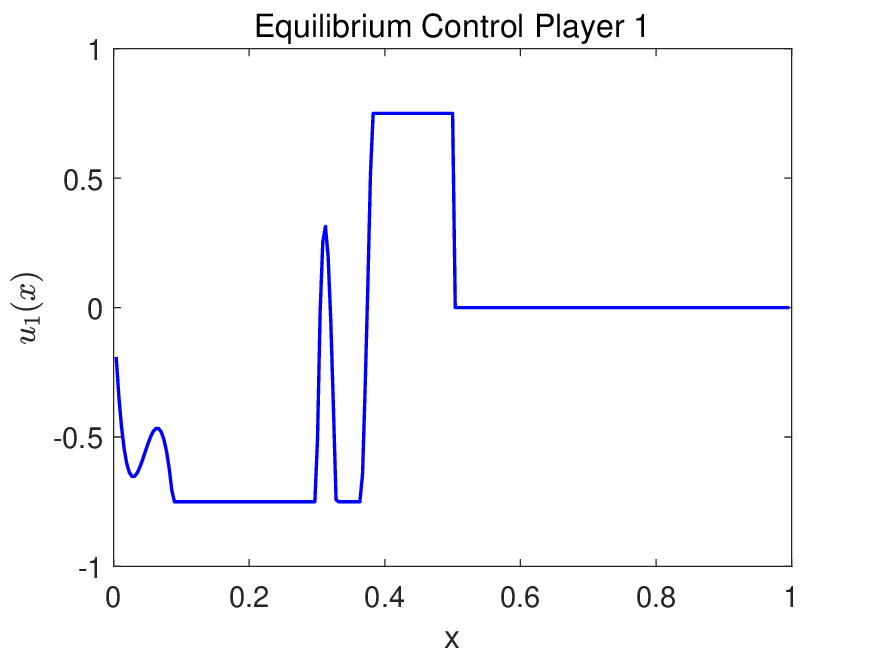}	
	 	\end{minipage}}
	 	\subfigure[Optimal control $u_2$]{
	 		\begin{minipage}[t]{4.0cm}
	 			\centering
	 			\includegraphics[width=\linewidth]{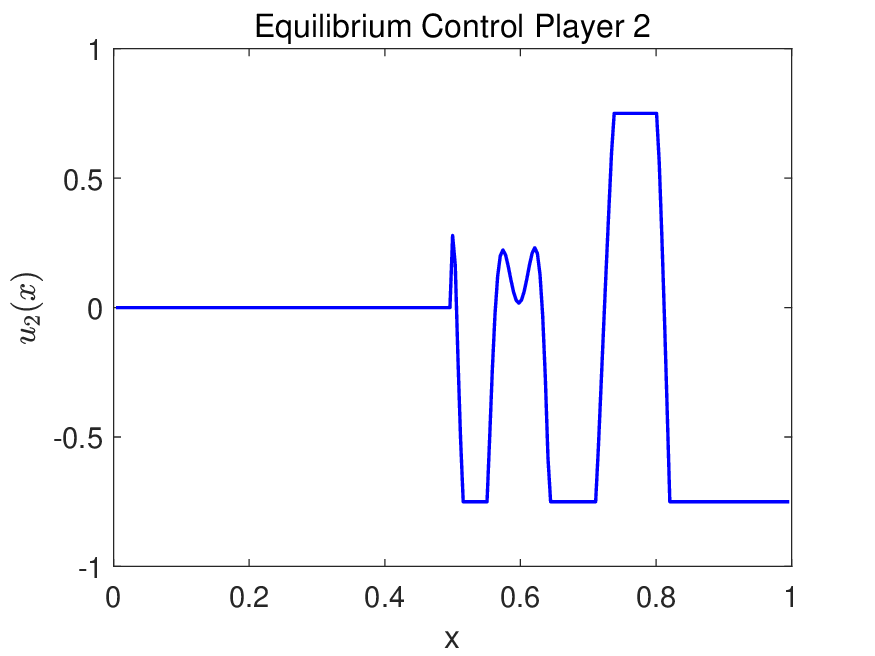}
	 	\end{minipage}}
	 	\subfigure[Sample states $y(\cdot,\omega)$]{
	 		\begin{minipage}[t]{4.0cm}
	 			\centering
	 			\includegraphics[width=\linewidth]{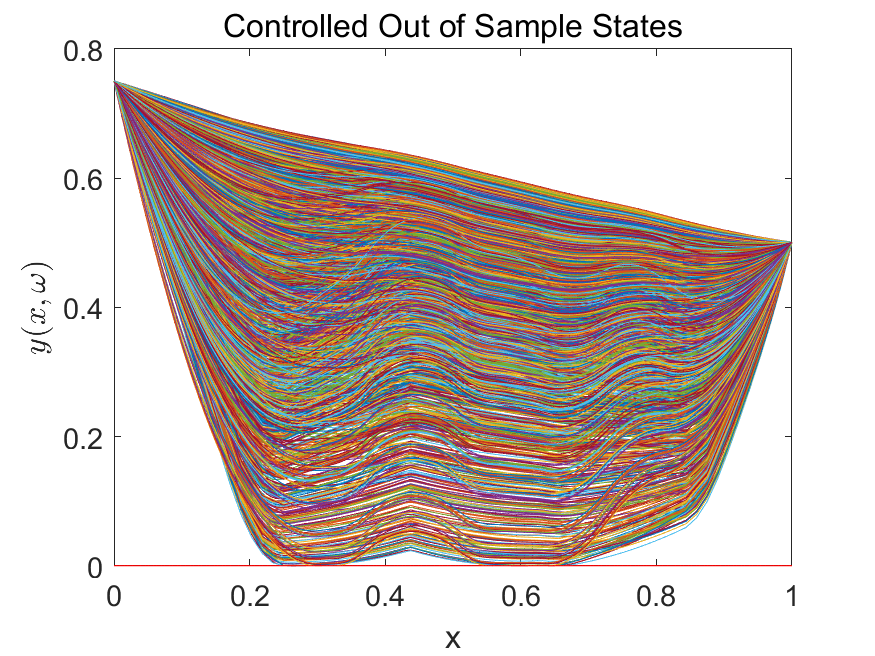}	
	 	\end{minipage}}
	 	\caption{\centering Numerical solution obtained by Algorithm \ref{adaptive-ALM} with $h=\frac{1}{2^8}$ of Example 2}
	 	\label{fig2}
	 \end{figure}
\section{Conclusion}\label{Sec: Con}
In this paper, we proposed a symmetric Gauss-Seidel based alternating proximal augmented Lagrangian method (sGS-APALM) for a class of monotone generalized Nash equilibrium problems with jointly linear constraints. Compared with existing regularization and ALM-type methods, the main advantage of the proposed method is that each iteration eliminates the need to solve complicated Nash equilibrium problems. Instead, it updates the players' strategies alternately by solving a sequence of unconstrained quadratic programs. This yields an efficient and practical iteration scheme. In contrast to existing splitting-based methods, our convergence analysis requires only the pseudo-gradient operator to be monotone and Lipschitz continuous, thereby relaxing the usual theoretical restrictions and preserving the simplicity of the scheme. We then focused on a class of risk-neutral PDE-constrained GNEPs and discussed the practical implementation. Finally, some preliminary numerical results demonstrated the efficiency and effectiveness of our proposed method.

\section*{Acknowledgment}
The authors wish to thank Professor Jong-Shi Pang for his insightful discussions and valuable suggestions on an early draft on this work.

 \bibliographystyle{siamplain}

\bibliographystyle{siamplain}

\end{document}

%% file: ex_shared.tex
\usepackage{lipsum}
\usepackage{amsfonts}
\usepackage{graphicx}
\usepackage{epstopdf}
\usepackage{algorithmic}
\usepackage{subfigure}
\usepackage{multirow}
\usepackage{amsmath}
\usepackage{mathtools}
\usepackage{amssymb}
\ifpdf
  \DeclareGraphicsExtensions{.eps,.pdf,.png,.jpg}
\else
  \DeclareGraphicsExtensions{.eps}
\fi

\newsiamremark{remark}{Remark}
\newsiamremark{hypothesis}{Hypothesis}
\crefname{hypothesis}{Hypothesis}{Hypotheses}
\newsiamthm{claim}{Claim}
\newsiamremark{fact}{Fact}
\crefname{fact}{Fact}{Facts}

\headers{A sGS-APALM for GNEPs in Banach spaces}{Defeng Sun, Hailing Wang, and Wei Zhao}

\title{A symmetric Gauss-Seidel based alternating proximal ALM for generalized Nash Equilibrium problems in Banach spaces\thanks{Submitted to the editors DATE.
\funding{The work of Defeng Sun was supported in part by the Research Center for Intelligent Operations Research, RGC Senior Research Fellow Scheme No. SRFS2223-5S02 and GRF Project No. 15309625.}}}

\author{Defeng Sun\thanks{Corresponding Author. Department of Applied Mathematics, The Hong Kong Polytechnic University, Hung Hom, Kowloon, Hong Kong
  (\email{defeng.sun@polyu.edu.hk}).}
\and Hailing Wang\thanks{Department of Applied Mathematics, The Hong Kong Polytechnic University, Hung Hom, Kowloon, Hong Kong 
  (\email{hailing.wang@polyu.edu.hk}).}
\and Wei Zhao\thanks{Department of Mathematical Sciences, Tsinghua University, Beijing, China; Department of Applied Mathematics, The Hong Kong Polytechnic University, Hung Hom, Kowloon, Hong Kong
  (\email{zhaow22@mails.tsinghua.edu.cn, wei22.zhao@polyu.edu.hk}).}}

\usepackage{amsopn}
